\documentclass{siamart1116}

\usepackage{verbatim}
\usepackage{lipsum}
\usepackage{amsfonts}
\usepackage{graphicx}
\usepackage{epstopdf}
\usepackage{algorithmic}
\usepackage{mathtools}

\usepackage[shadow]{todonotes}
\usepackage{dsfont}
\ifpdf
  \DeclareGraphicsExtensions{.eps,.pdf,.png,.jpg}
\else
  \DeclareGraphicsExtensions{.eps}
\fi

\numberwithin{equation}{section}

\newcommand{\bbR}{\mathbb{R}}
\newcommand{\OOmega}{D}
\newcommand{\ep}{\varepsilon}
\newcommand{\sN}{\mathcal{N}}
\newcommand{\sC}{\mathsf{C}}
\newcommand{\GGamma}{\Sigma}
\newcommand{\one}{\mathds{1}}
\newcommand{\ue}{u^{\varepsilon}}
\newcommand{\be}{\begin{equation}}
\newcommand{\ee}{\end{equation}}
\newcommand{\beqas}{\begin{eqnarray*}}
\newcommand{\eeqas}{\end{eqnarray*}}

\newcommand{\eproof}{{\hfill$\Box$}}
\newcommand{\dee}{\mathrm{d}}
\def\IR{\mathbb{R}}

\newcommand{\dd}{ \,\mathrm{d}}
\newcommand{\p}{\partial}
\def\bbb[#1]{\boldsymbol{#1}}
\def\mmm[#1]{\mathsf{#1}}

\numberwithin{theorem}{section}
\newtheorem{remark}{Remark}
\newcommand{\TheTitle}{Reconciling Bayesian 
and perimeter regularization  for binary inversion} 
\newcommand{\TheAuthors}{O. R. A. Dunbar, M. M. Dunlop, C. M. Elliott, V. Ha Hoang, and A. M. Stuart}

\headers{Bayesian methods for binary inversion}{\scriptsize\TheAuthors}

\title{{\TheTitle}\thanks{Submitted to the editors DATE.
\funding{The research  of CME was partially supported by the Royal Society via a Wolfson Research Merit Award; the work of AMS by DARPA contract
 W911NF-15-2-0121; the 
work of ORAD, CME and AMS by the EPSRC programme grant EQUIP; the work of ORAD by the NSF grant AGS‐-1835860;
the work of MMD and AMS 
by AFOSR Grant FA9550-17-1-0185 and ONR Grant N00014-17-1-2079;
the work of MMD by the EPSRC MASDOC Graduate Training Program; 
VHH gratefully acknowledges the MOE AcRF Tier 1 grant RG30/16.}}}

\author{
  Oliver R. A. Dunbar\thanks{Geological \& Planetary Sciences, California Institute of Technology, USA, 91125 (\email{odunbar@caltech.edu}).}
  \and
  Matthew M. Dunlop\thanks{Courant Institute of Mathematical Sciences, New York University, USA, 10012 (\email{matt.dunlop@nyu.edu}).}
  \and
  Charles M. Elliott\thanks{Mathematics Institute, University of Warwick, UK, CV4 7AL (\email{c.m.elliott@warwick.ac.uk}).}   
  \and \newline
  Viet Ha Hoang\thanks{Division of Mathematical Sciences, School of Physical and Mathematical Sciences, 
  Nanyang Technological University, Singapore 637371 (\email{vhhoang@ntu.edu.sg}).}
  \and
  Andrew M. Stuart\thanks{Computing \& Mathematical Sciences, California Institute of Technology, USA, 91125 (\email{astuart@caltech.edu}).}
}

\ifpdf
\hypersetup{
  pdftitle={\TheTitle},
  pdfauthor={\TheAuthors}
}
\fi

\begin{document}

\maketitle

\begin{abstract}
A central theme in classical algorithms for the reconstruction of 
discontinuous functions from observational data is perimeter 
regularization via the use of the total variation.
On the other hand, sparse or noisy data often demands a probabilistic 
approach to the reconstruction of images, to enable uncertainty quantification;
the Bayesian approach to inversion, which itself introduces
a form of regularization, is a natural framework in which
to carry this out. In this paper the link between Bayesian inversion methods
and perimeter regularization is explored. 
In this paper two links are studied: (i) the maximum a posteriori
(MAP) objective function of
a suitably chosen Bayesian phase-field approach is shown to be closely
related to a least squares plus perimeter regularization objective;
(ii) sample paths of a suitably chosen Bayesian level set formulation are 
shown to possess finite perimeter and to have the ability to
learn about the true perimeter. 

\end{abstract}



\section{Introduction}

\subsection{Problem Statement}
Let $\OOmega$ be the unit cube $(0,1)^d\subset\IR^d, d=2,3$. 
Let $K:L^1\bigl(\OOmega\bigr)\to \IR^J$ be a bounded linear operator.
We consider the problem of recovering a binary-valued function 
$u \in BV_{\rm binary}(\OOmega)$ where 
$$BV_{\rm binary}(\OOmega)=\{\psi\in BV(\OOmega):\ \psi(x)\subset\{\pm 1\}, x\in \OOmega\}$$
from finite dimensional data $y\in \mathbb R^J$ satisfying
\begin{equation}
\label{eq:basiceq}
y=Ku +\ep^c\eta.
\end{equation}
Here the  finite number of observations are corrupted by noise $\ep^c\eta$ of size $\ep^c$ for which we assume that 
$\eta$ is a centred Gaussian
$\sN(0,\Sigma)$ with positive definite covariance  $\Sigma\in \mathbb R^{J\times J}$.  Here $\ep$ and $c$ denote constants, with $\ep \ll 1$ and $c>0$ (small noise)
or $c=0$ (noise on the order of the observations). Observe that from an application perspective the 
space $BV_{\rm binary}(\OOmega)$ is a natural model for binary images. The problem of determining $u \in BV_{\rm binary}(\OOmega)$
from $y$ given by \eqref{eq:basiceq}
thus constitutes a canonical binary  inverse problem for which
the issue is to recover the interface between the different domains in $\OOmega$ in which function $u$  takes its two values. 
For a given $u$ a measure of the discrepancy with the data is the following scaled misfit functional
\begin{equation}
J(u):=\frac{1}{2\ep^{2c}}\Bigl|\GGamma^{-\frac{1}{2}}(y-Ku)\Bigr|^2
\end{equation}
A typical deterministic approach to recover of $u$, based on this
misfit, would be to specify a model prior space $\mathcal P$  for $u$ 
and to regularize the misfit functional by addition of a 
functional $\mathcal R(u)$ defined on $\mathcal P$,
and then to seek a solution to the optimization problem
\begin{equation}
\inf_{u\in \mathcal P}\Big(J(u)+\mathcal R(u)\Big).
\end{equation}
An intuitive and common method of regularization for binary problems is to penalize the perimeter of the 
interface. In the case of $\mathcal P=BV_{\rm binary}(\OOmega)$
this leads to
\begin{equation}\label{perimreg}
\inf_{u\in BV_{\rm binary}(\OOmega)}\Big(J(u)+\sigma \int_\OOmega|\nabla u|\Big)
\end{equation}
where $\int_\OOmega|\nabla u|$ denotes the total variation of $u$ and $\sigma >0$ is a parameter. In this way, the minimization over $BV_{\rm binary}(\OOmega)$ identifies perimeter regularization with total variation \cite{rudin1992nonlinear} and the Mumford-Shah \cite{MumSha89} approaches, since these 
regularizations coincide on binary functions.

Often perimeter regularization is relaxed to a convex 
regularization  by allowing values of $u \in [-1,1]$. 
An alternative, non-convex, relaxation is to use a Cahn-Hilliard functional to approximate the perimeter functional. For example one might consider, 
for a small parameter $\tilde\ep>0$,
\begin{equation}
\inf_{W^{1,p}(\OOmega)}\Big(J(u)+\sigma_W \int_\OOmega \Big( \tilde \ep|\nabla u|^2+\frac{1}{\tilde \ep}W(u)\Big)dx\Big)
\end{equation}
where $W(\cdot)$ is a double well potential. The minimizers of this Cahn-Hilliard functional are known as phase-fields, and this relaxation is often referred to as a phase-field regularization. In appropriate circumstances this $\Gamma-$converges to (\ref{perimreg}) in the limit $\tilde\ep\to 0$.

However since the unknown observational noise $\eta$ has an assumed  Gaussian 
distribution, it is natural 
to take a probabilistic approach to the recovery of $u$ and 
model uncertainty about $u$, and hence the interface 
between different domains, through a probability distribution. 
This leads to Bayesian formulations of the problem in which a prior
probability distribution is specified on the unknown function,
and the likelihood of the data is used to compute a posterior probability
distribution on the unknown function, given the data. The prior probability
distribution imposes a prior space ${\mathcal Q}$ where, almost surely,
samples from the posterior distribution live; the mean or mode of the
posterior distribution will typically live in a smoother space ${\mathcal P}
\subset {\mathcal Q}$; this space ${\mathcal P}$ will be analogous to the
prior space ${\mathcal P}$ described above.

We adopt two approaches. In the first, the level set method, we reformulate the inverse problem 
as determining smooth functions $v$ whose zero level set defines the interface in the unknown
function $u$. Specifically the sign of $v$ defines $u$ and the prior probability
distribution on $v$ yields a space ${\mathcal Q}$ containing $C^1$ functions.
The pushforward measure on $u$, defined by the sign function, has support in $BV_{\rm binary}(\OOmega)$.
In the second approach, motivated by phase-field regularization, we relax
the prior measure to allow for smooth functions $u$ with sharp interfaces near zero;
the implied space ${\mathcal Q}$ for functions $u$ contains
$H^s$ functions for any $s<2-d/2$, whilst  ${\mathcal P}$ contains $H^2$ functions.

We define a prior distribution over the smooth function $v$
or over $u$,
and formulate an associated likelihood determined by $J(u)$.  Bayes' theorem is then employed 
in a form which implies that the posterior probability 
measure is absolutely continuous with respect to the prior probability measure. Taking the sign of such distributions on $v$ yields a distribution for $u$. 
By sampling the distribution one can approximate the mean. It is then interesting to make a connection between this
mean and the solution to the perimeter regularization problem \eqref{perimreg}. 

The main goals of the paper are to investigate how the connection to perimeter 
regularization appears for different (Bayesian) formulations of the
inverse problem, and to demonstrate the performance and applicability of 
these formulations for both the linear inverse problem \eqref{eq:basiceq},
and nonlinear generalizations.

\subsection{Background}
There are many problems in the physical sciences where piecewise
constant reconstruction is of interest, for example in subsurface
inversion and imaging \cite{CW,EnKF_level,Dorn,ILS14,lee2013bayesian,cardiff2009bayesian}  and other
problems in the physical sciences \cite{0266-5611-25-12-125001};
the problem of image deblurring \cite{hansen2006deblurring} (in particular, for barcodes and QR codes \cite{choksi2010,jin2016,van2015regularization,soros2015fast,iwen2013symbol,rioux2019blind}) is also
of interest in the context of piecewise
constant reconstruction. We draw our motivation from these problems
and our numerical experiments are based on
imaging problems possessing a variety of geometric interfaces,
smooth and including edges.

A seminal paper linking probabilistic approaches to classical
numerical methods is \cite{diaconis1988bayesian}, and a review
describing developments since then can be found in \cite{cockayne2019bayesian}.
In the context of inverse problems the link between Bayesian and
classical approaches is well-understood in the setting of Gaussian random
field priors: the Bayesian maximum a posteriori (MAP)
estimator \cite{kaipio2006statistical,DLSV13} is
then the solution of a Tikhonov-Phillips regularized least squares problem
\cite{engl1996regularization}.
When more complex priors are used the connection between classical
and Bayesian perspectives is more subtle, even for linear
inverse problems 
\cite{burger2014maximum,helin2011hierarchical, calvetti2007gaussian,calvetti2008hypermodels,lassas2009discretization}; see
\cite{agapiou2018sparsity,clason2018generalized,helin2015maximum}
for recent work generalizing \cite{DLSV13} beyond the Gaussian prior setting. 
Two interesting approaches to Bayesian inversion, both using thresholding
as we  do in this paper, may be found in
\cite{niemi2015dynamic} and \cite{hosseini2017well}.
In the one dimensional setting an interesting construction of
random functions with finite TV norm may be found in
\cite{cohen1997nonlinear}; a Poisson process is used to define
points of discontinuity, with smooth processes between these
discontinuity points.

In interface reconstruction, 
classical methods have been dominated
by inversion techniques which penalize the length of the perimeter
between different subdomains. Two contrasting approaches for describing  interfaces are the use of a level set of a continuous function or a characteristic function which takes just two values. Total variation (TV) regularization has played a central role
\cite{rudin1992nonlinear} and has been shown to lead to
empirically effective methods which are computationally efficient.
The phase-field representation of interfaces is
described in \cite{DecDziEll05}. The method approximates the perimeter using a scaled gradient energy and double well potential for which minimizers have diffuse interfaces of width a small length scale and which encloses a zero level set of the minimizer.  In contrast, the level set 
approach of \cite{sethian1999level} represents interfaces as 
level sets of continuous fields. See \cite{brett2014phase,deckelnick2016double,santosa1996level} for
applications of phase-field and level-set ideas in classical, 
non-Bayesian, inversion for interfaces. 
For simplicity this paper focuses primarily on recovery of a binary
function, taking two known values, with unknown interface. In the more 
general setting of recovering unknown piecewise continuous function, in which
the interface and the values of the function off the interface, are unknown,
the classical TV and Mumford-Shah perimeter regularization methods become 
distinct.  For an elliptic problem, \cite{ChaTai04} use TV regularization 
on a level set function to penalize both perimeter length and jump 
discontinuities; the Mumford-Shah minimization can be written over a 
suitable space to jointly minimize the function and its set of 
discontinuity \cite{RamRin10}.

Computer power has started to render Bayesian
inversion techniques tractable in some applications 
\cite{kaipio2006statistical,stuart2010inverse,dashti2013bayesian}, 
enabling uncertainty quantification. 
In this paper we address the question of how perimeter regularization appears
within Bayesian inversion techniques for the reconstruction of
binary function $u$. This is a notoriously difficult problem,
as made transparent in the paper \cite{lassas2004can} which showed that
use of discrete total variation regularization, in a Bayesian setting,
does not lead to a meaningful problem in the continuum limit;
this work led to the development of new Besov priors in
\cite{lassas2009discretization}, and the combination of TV and Gaussian 
priors \cite{yao2016tv}. Other approaches, with demonstrable numerical success, 
make use of the introduction of 
hyperparameters \cite{lee2013bayesian,calvetti2019hierachical}.
Instead of approximating the TV regularization, one can derive Bayesian approaches based on the  Mumford-Shah functional
\cite{helin2011hierarchical}; these methods can be extended to 
higher order functionals such as Blake-Zisserman \cite{carriero2015survey}.
Our probabilistic level set based method
generalizes to a hierarchical Bayesian approach that learns the 
unknown continuous function off the unknown interfaces, as 
well as the interface itself; the interface may be viewed as a nonparametric
hyperparameter.

\subsection{Our Contribution}

In detail our contributions are as follows:

\vspace{0.1in}

\begin{enumerate}
\item We formulate a Bayesian level set approach and establish conditions under which this 
leads to posterior samples with almost surely finite perimeter, 
and, hence, almost surely finite total variation (TV) norm. This  demonstrates 
that TV regularization arises naturally out of appropriately
chosen Bayesian formulations of inversion. 

\item We formulate a Bayesian based phase-field approach and establish 
a link with perimeter regularization through its MAP estimator.
We prove, for appropriate choice of prior distribution
and parameters carefully scaled with respect to $\ep$ (and so not strictly Bayesian), that 
the maximum a posteriori (MAP) estimator for this phase-field approach has a
 $\Gamma-$limit as $\ep \to 0$. This limit is exactly the perimeter (TV) regularization of the least squares fidelity objective function.

\item For a linear inverse problem we provide numerical investigations 
of these approaches; we also compare to (widely
used)  Gaussian process regression which is a natural method
in this linear setting. These investigations  demonstrate that
the level set approach may be implemented quite cheaply in
comparison with the phase-field approach, for similar levels
of reconstruction accuracy. Also it is demonstrated that
the level set approach
can learn the true perimeter. Gaussian process regression also 
performs well at a low computational cost for the linear problem,
but is not readily extended to nonlinear problems.

\item We provide numerical evidence for the flexibility of the
Bayesian level set approach, by showing an application to a nonlinear inverse problem
arising from the eikonal equation. Within this context, we also show 
that hyperparameters contained in the statistical model
may also be efficiently learnt. 

\end{enumerate}

\vspace{0.1in}

\subsection{Some Notation}

\label{ssec:Not}

We use $|\cdot|$ to denote the Euclidean norm on $\mathbb R^J$. Let $C_\#^{k,\gamma}(\bar\OOmega), k\ge 0$ denote the space of real valued continuous periodic functions
 on $\bar\OOmega$ whose derivatives up to order $k^{\rm th}$
derivative are H\"older  continuous with exponent $\gamma$. By virtue of continuous embedding 
$K$ is  a bounded linear operator on $C_\#^{k,\gamma}(\bar\OOmega)$
for any integer $k\ge 0.$ 
Also   let $H^k_\#(D), k\ge 0$ denote the restriction 
to periodic functions of the Sobolev space of $H^k(D)$ of $k$-times weakly differentiable real-valued functions on $D$. These Sobolev spaces are
readily characterized as weighted $\ell^2$ spaces on Fourier
coefficients \cite{robinson}.
Let $X$ denote the space 
$C(\bar\OOmega)$, restricted to periodic functions, and let $H$
denote $L^2(\OOmega).$ 

Fix constants $\delta>0,  \tau>0, q \ge 0$ and $a_i\ge 0, i=1,2,3$. Denote
$\tilde \delta=(\delta,\tau,q)$ and $\tilde a=(a_1,a_2,a_3)$. 
We define a covariance operator   $\mathcal C_{\ep,\tilde \delta,\tilde a}$  implicitly 
as the solution operator corresponding to the weak formulation of the
following elliptic boundary value problem:
given $f \in H$ find $\eta \in  H^2_\#(\OOmega),$ so that
\begin{align}
\label{eq:cinv_def}
~ \delta\ep^{-2a_1}\Delta^2 \eta-q\delta\ep^{-2a_2}\Delta \eta+\tau^2\delta\ep^{-2a_3}\eta= f. 
\end{align}
Elliptic regularity gives $\eta \in H^4_\#(\OOmega)$ and so we may
define $(\mathcal C_{\ep,\tilde \delta,\tilde a})^{-1}:H^4_\#(\OOmega)\rightarrow H$ by
$(\mathcal C_{\ep,\tilde \delta,\tilde a})^{-1}\eta =f$ for $f\in H.$ 
The Hilbert  space  $\mathcal E_{\ep,\tilde \delta,\tilde a}$,  is defined to be $H^2_\#(\OOmega)$ endowed with the  norm  
$$\|\xi\|_{\mathcal E_{\ep,\tilde \delta,\tilde a}}^2: =\langle (\mathcal C_{\ep,\tilde \delta,\tilde a})^{-\frac{1}{2}}\xi, (\mathcal C_{\ep,\tilde \delta,\tilde a})^{-\frac{1}{2}}\xi \rangle =\delta\int_{\OOmega}
 \bigl(\ep^{-2a_1}|\triangle \xi|^2+q \ep^{-2a_2}|\nabla \xi|^2 +\tau^2\ep^{-2a_3}\xi^2\bigr)\dee x$$
where $\langle \cdot, \cdot \rangle$ denotes the standard $L^2(\OOmega)$
inner-product. By polarization an inner-product is defined on 
$\mathcal E_{\ep,\tilde \delta,\tilde a}$. 
The three parameters $\delta, q$ and $\tau$ weight the contributions
of the $ H^2_\#(\OOmega),  H^1_\#(\OOmega)$ and $L^2(\OOmega)$ terms
appearing in the Hilbert  space  $\mathcal E_{\ep,\tilde \delta,\tilde a}$ norm whereas 
the parameters $a_1, a_2$ and $a_3$ scale these terms with respect to
powers of $\ep$. The Hilbert space $\mathcal E_{\ep,\tilde \delta,\tilde a}$ is exactly the Cameron-Martin space for the Gaussian $\sN(0,\mathcal C_{\ep,\tilde \delta,\tilde a})$, \cite[Definition 6.26]{dashti2013bayesian}.
In the following we write $C$ and $\mathcal E$, with the dependence on the parameters being understood. In the computations it is made clear which values of the parameters are chosen.

\begin{remark}We note that including an $H^2_\#(\OOmega)$
contribution in the Cameron-Martin norm is required in dimensions $d=2,3$
in order to ensure that the underlying Gaussian is supported on
continuous functions. 
In dimension $d=1$ it is possible to remove this 
requirement. \cite{Sivak}.\end{remark}

\begin{remark} The requirement that $\tau>0$ is made to ensure that  $(\mathcal C_{\ep,\tilde \delta,\tilde a})^{-1}$
is invertible on $L^2_\#(\OOmega)$. This could also be addressed 
when $\tau=0$ by working on spaces of functions where the mean-value
is zero.   \end{remark}

\subsection{Outline Of The Paper}

In \cref{sec:B} we formulate three inversion approaches for the linear inverse problem
\eqref{eq:basiceq}, a Bayesian level set based approach, a Bayesian based phase-field  regularization and Gaussian process regression. We  introduce 
level set and phase-field priors, both of which are 
non-Gaussian  and state a well-posedness result for the 
resulting posterior distributions, as well as some relevant properties. 
\Cref{sec:M} characterizes the MAP estimator for the phase-field 
prior. 
Under appropriate parameter scalings, we obtain perimeter regularization as a $\Gamma-$limit for the
MAP estimator in the small noise regime, using the analysis in \cite{hil}.
This $\Gamma-$limit links the MAP estimator to classical deterministic perimeter regularization.
In \cref{sec:N} we describe testing the approaches with numerical experiments based on MCMC.  
 We also discuss the properties of
the length of level sets of Gaussian random fields, and of random fields
whose law has density with respect to a Gaussian random field, and use this
to demonstrate that the level set approach {\it learns} the perimeter.  
These numerical results establish interesting links between Bayesian 
level set inversion and perimeter regularization.  
Within \cref{sec:eik} we go beyond linear inverse problems, 
showing that the Bayesian level set approach readily extends
both to a nonlinear inverse problem arising from the eikonal equation,
and to the learning of unknown hyperparameters from the
prior; for hyperparameter learning we use algorithms introduced
in \cite{DunIglStu17} and further developed in \cite{CheDunPapStu19_preprint}. 
 \cref{sec:A} contains proofs of the main results relating to MAP estimators for the phase-field approach.

\section{Inversion Approaches}

\subsection{A Bayesian Level Set Based Approach To The Inverse Problem}
\label{sec:B}

\subsubsection{Prior And Likelihood} \label{ssec:P}

Let $\one_{\cdot}$ denotes the characteristic function of a set.  
Our model prior for $u$ is that   
$$u=\one_{D_+}-\one_{D_+^c}$$
where $D_+\subset D$ is a random set defined in such a way that
that ${\rm leb}({\overline D_+}\backslash D_+)=0$, almost surely;
this ensures that $u$ is almost surely in $BV_{\rm binary}$. 
 This is achieved  by working with an auxiliary variable $v$ 
and recovering $u$ by an application
of a thresholding (sign) function $S:\bbR \mapsto \{-1,0,+1\}$   defined by
$$S(v)=1, \;v >0, \quad S(0)=0  \quad{\rm and}\quad S(v)=-1, \; v<0.$$
The  prior on $v$ is chosen to be $\mu_{0,\alpha}=\sN(0, C^{\frac{\alpha}{2}})$, with the power satisfying $\alpha>d/2$. 
This  implies almost sure continuity
of $v$. The prior for $u$ is then defined by the pushforward of $\mu_{0,\alpha}$ under $S$
\begin{equation}\label{eq:ufv}
 \mathcal E^S:=\{u=S(v)~|~v\in \mathcal E\}.
 \end{equation}
This is justified since the level sets of the Gaussian
random field $v$ have Lebesgue measure zero \cite[Proposition 2.5]{igl}. 
Also by the lemma which follows, it holds almost surely
that if $\alpha>1+d/2$ then $u \in \{\pm 1\}$ $\OOmega$ a.e.  
and has bounded total variation. 
\begin{lemma}
\label{lem:per}
If function $v$ is drawn from measure $\mu_{0,\alpha}$ with
$\alpha>1+d/2$ then almost surely function $u$ defined by
\eqref{eq:ufv} has finite total variation norm. 
\end{lemma}

\begin{proof}
If $\alpha>1+d/2$ then almost surely $v \sim \mu_{0,\alpha}$ 
will be a $C^1$ function.
The paper \cite{marie2006level} establishes that the level set
$v=0$ will then have finite length, almost surely. Since $u$
is a binary function given by \eqref{eq:ufv} this establishes that
$u$ will have finite total variation, almost surely.
\end{proof}

If we set $u=S(v)$ it follows that model  \cref{eq:basiceq} becomes
$$y=KS(v)+\ep^c\eta$$
and hence that $y|v$ is distributed as the 
Gaussian $\sN(KS(v),\ep^{2c}\GGamma)$. The likelihood is given by the Gaussian density proportional to 
\[
\exp(-J(S(v))) = \exp\left(-\frac{1}{2\ep^{2c}}\Bigl|\GGamma^{-\frac{1}{2}}\bigl(y-KS(v)\bigr)\Bigr|^2\right).
\]
$\Phi(v;y)=J\circ S(v)$ is the negative log-likelihood function.

\subsubsection{Posterior}

In the following proposition we to establish a relationship between the 
(posterior) distribution $\mu^y$ of the 
random variable $v|y$, the prior on $v$ and the likelihood on $y|v$,
by means of an infinite dimensional Bayes' Theorem, \cite{dashti2013bayesian}: 

\begin{proposition}
\label{prop:12}
Let $\alpha>d/2.$ Then
the posterior probability $\mu^y$ on random variable $v|y$
is a probability measure supported on 
$C_\#^{k,\gamma}(\bar\OOmega)$
for all $\gamma \in [0,\gamma')$, where $k \in \{0,1,2,\cdots\}$ is 
chosen so that
$\gamma':=\alpha-\frac{d}{2}-k \in (0,1],$
and is determined by
\[
\frac{\dee\mu^y}{\dee\mu_{0,\alpha}}=\frac{1}{Z}
\exp\left(-\frac{1}{2\ep^{2c}}\Bigl|\GGamma^{-\frac{1}{2}}\bigl(y-KS(v)\bigr)\Bigr|^2\right);
\]
here $Z \in (0,\infty)$ is the normalization constant that makes 
$\mu^y$ a probability measure.  
\end{proposition}
\begin{proof}
This follows from an application of the theory in \cite{igl}.
\end{proof}

We state some beneficial properties of the the posterior, making it fit for purpose in this 
application. Specifically, the posterior $\mu^y$ has a continuous dependence on $y$, and the 
pushforward under $S$ defines an implied posterior $\nu^y$, whose samples have finite total variation.

Recall the Hellinger distance between measures $\mu$ and $\mu'$, 
defined  with respect to any common reference  measure $\mu_0$ (but independent of it)
and given by 
\[
d_{\rm hell}(\mu,\mu')=\sqrt{\left(\int_X\left(\sqrt\frac{\dee\mu}{\dee\mu_0}-\sqrt\frac{\dee\mu'}{\dee\mu_0} \right)^2 \dee\mu_0\right)}.
\]

\begin{proposition}
\label{prop:12b} 
Given the setting of \cref{prop:12}
(i) the posterior measure $\mu^y$ 
is locally Lipschitz continuous with respect to $y~\in~\IR^J$; more precisely: 
if $|y|< \rho$ and $|y'|< \rho$ for a constant $\rho>0$ then there is a constant $C=C(\rho)$ such that
\[
d_{\rm hell}(\mu^y,\mu^{y'})\le C(\rho)|y-y'|;
\]
(ii) if $\alpha>1+d/2$ then $u=S(v)$, with $v \sim \mu^y$, has
finite total variation norm, almost surely.
\end{proposition}

\begin{proof}
(i) follows from application of the theory in \cite{igl}. (ii) follows by noting that, since $\mu^y$ has density
with respect to $\mu_{0,\alpha}$, anything which holds almost 
surely under $\mu_{0,\alpha}$, will also hold almost surely
under $\mu^y$. Application of \cref{lem:per} gives the
desired result.
\end{proof}

\subsection{A Phase-Field Regularization Based Bayesian Approach To The Inverse Problem}
\label{sec:PF}

The Bayesian level set approach of Section \ref{sec:B} is formulated in
terms of a prior on a smooth variable $v$ whose push forward under the
thresholding map gives a function $u$ taking values in $\{-1,1\}$, respecting
the fact that the data takes values of the
form $\{-1,1\} ~+ ~\mbox{noise}$.
Here we describe a different approach, one in which the $\ep-$dependent prior on $u$ may take
values anywhere in $\mathbb{R}$, but concentrates close to $\{-1,1\}$
when $\ep \ll 1.$ This leads to a connection with phase-field regularization.

\subsubsection{Prior And Likelihood}

Fixing constants $r,b>0$ define
$\Psi:X \to \IR^+$ by
\begin{equation}
\label{eq:psi}
\Psi(u)=\frac{r}{\ep^b}\int_{\OOmega}\frac{1}{4}\bigl(1-u(x)^2\bigr)^2\,\dee x.
\end{equation}
We define the prior probability measure $\nu_0$ on $X$
via the Radon-Nikodym derivative
\be
\frac{\dee\nu_0}{\dee\mu_0}=\frac{1}{Z_{0}}\exp\left(-\Psi(u)\right).
\label{eq:nu02}
\ee
where $\mu_0$ is the   Gaussian measure $\mu_0=\sN(0,C)$ on 
the Hilbert space $H$.The normalization $Z_{0}$ is chosen so that $\nu_0$ is a probability measure.
Since the Gaussian measure $\mu_0$ is supported on continuous functions in dimensions
$2$ and $3$,  so is the non-Gaussian measure $\nu_0$. Furthermore,
since $r,b>0$ and $\ep \ll 1,$ this measure will concentrate on functions taking
values close to $\pm 1.$ In what follows the choice of parameter $b$ will be crucial, and will be explained
below; the precise value of the positive parameter
$r$ is less significant.

The random variable $y|u$, 
given by \cref{eq:basiceq}, has a Gaussian distribution 
$\sN(Ku,\ep^{2c}\GGamma)$ 
and 
the likelihood is the Gaussian density proportional to
\[
\exp(-J(u)) = \exp\left(-\frac{1}{2\ep^{2c}}\Bigl|\GGamma^{-\frac{1}{2}}\bigl(y-Ku)\bigr)\Bigr|^2\right).
\]
$\Phi(u;y)=J(u)$ is the negative log-likelihood function.

\subsubsection{Posterior}
\label{ssec:post}
We let $\nu^y(\dee u)$ denote the probability of the conditioned
random variable $u|y$.The following propositions are the analogue of \cref{prop:12} and 
\cref{prop:12b}. They are proved by a straightforward application of the 
theory in \cite{dashti2013bayesian,igl}: 
\begin{proposition}
\label{prop:1}
The posterior probability $\nu^y$ on random variable $u|y$ is a probability measure 
supported on $C_\#^{k,\gamma}(\bar\OOmega)$ for any $\gamma<2-d/2$
and determined by
\[
\frac{\dee\nu^y}{\dee\nu_0}=\frac{1}{Z}
\exp\left(-\frac{1}{2\ep^{2c}}\Bigl|\GGamma^{-\frac{1}{2}}(y-Ku)\Bigr|^2\right),
\]
where $Z \in (0,\infty)$ is the normalization constant that makes 
$\nu^y$ a probability measure. 
\end{proposition}

\begin{proposition}
\label{prop:1b}

In the setting of \cref{prop:1}, the posterior measure $\nu^y$ 
is locally Lipschitz continuous with respect to $y~\in~\IR^J$; more precisely: 
if $|y|< \rho$ and $|y'|< \rho$ for a constant $\rho>0$ then there is a constant $C=C(\rho)$ such that
\[
d_{\rm hell}(\nu^y,\nu^{y'})\le C(\rho)|y-y'|.
\]
\end{proposition}

\begin{remark}
For computations, it is convenient to draw samples from the Gaussian 
prior  $\mu_0$, not the prior $\nu_0$; to this end we note that
the posterior $\nu^y$ may be written as
\[
\frac{\dee\nu^y}{\dee\mu_0}=\frac{1}{Z_0Z}\exp\left(-\frac{r}{\ep^b}\int_\OOmega\frac{1}{4}\bigl(1-u(x)^2\bigr)^2\,
\dee x -\frac{1}{2\ep^{2c}}\Bigl|\GGamma^{-\frac{1}{2}}\bigl(y-Ku\bigr)\Bigr|^2
\right),
\] for normalization constants $Z_0,Z\in(0,\infty)$.
\end{remark}

\subsection{A Gaussian Process Regression On The Inverse Problem}
\label{sss:gp}

A popular approach for linear inverse problems
is to use a Gaussian process (GP) 
regression to find posterior parameter 
distributions \cite{book:Owh19,book:RasWil06}; this
maybe combined with thresholding to perform classification
\cite{book:RasWil06}, an approach we adapt here to learning
a binary function. 

A Gaussian process is a collection of random 
variables, with all finite subsets being described 
by a joint Gaussian distribution. Adapted to our specific
inverse problem, Gaussian process regression
proceeds by imposing a Gaussian prior on the unknown function, and
then conditioning this on $y$ given by \eqref{eq:basiceq}.
We take as prior $\mu_0$ the Gaussian $\sN(0,C).$
The posterior is completely described as $\nu^y = \sN(m_y,C_y)$, with mean $m_y$ and covariance function $C_y$. This is attractive as with a closed form for $m_y$ and $C_y$, we may sample the posterior directly without the need for MCMC, thus it is very computationally efficient. Closed forms are found to be
\[
m_y  = C K^*(\ep^{2c} \Sigma + KCK^*)^{-1}y,\quad C_y = C - C K^*(\ep^{2c} \Sigma + KCK^*)^{-1}KC.
\]
In the linear setting $m_y$ is the MAP estimator of the posterior, thus is the unique minimizer of the the functional 
\begin{equation}\label{eq:GPfunct}
J(u) + \mmm[R](u)= \frac{1}{2\ep^{2c}}\Big|\Sigma^{-\frac{1}{2}}(y-Ku)\Big|^2
+\frac{1}{2}\|u\|_{\mathcal E}^2.
\end{equation}
This approach is unnatural from a modelling point of view, as the difference $y-Ku$ appearing in \eqref{eq:GPfunct} contains comparison between data produced from a binary field and the forward map of a non-binary field. It is therefore difficult to interpret the results from this approach. 
Nonetheless we proceed in a fashion standard in machine learning, namely to threshold
the solution of the regression to obtain a classifier; in our particular
setting this corresponds to application of $S$ to $m_y$, or to samples
from the Gaussian posterior distribution.
We also note that the Gaussian process regression methodology
is very specific to the linear inverse problems and does
not extend directly to nonlinear forward mappings.

\section{MAP Estimators, Phase Field Regularization And $\Gamma$-convergence}
\label{sec:M}
 A MAP estimator of a Bayesian posterior distribution
maximizes the posterior probability. Intuitively, the MAP estimator
locates points in $X$ at which arbitrarily small balls will have
maximal probability. It is defined as follows, \cite{DLSV13}.
\begin{definition}\label{def:MAP}
A point $\bar z\in X$ is a MAP estimator for the posterior measure $\nu^y$ if
\[
\lim_{\rho\to 0}\frac{\nu^y(B^\rho(\bar z))}{\nu^y(B^\rho(z^\rho))}=1
\]
where \[
z^\rho=\underset{z\in X}{\rm argmax}\,\nu^y(B^\rho(z))
\]
and  $B^\rho(z),\rho>0$ is the ball centred at $z\in X$ with radius $\rho$.
\end{definition}
 
 We explore this concept in the context of phase-field regularization.
 Set $\Phi:X \times \IR^J \to \IR^+$ to be the sum 
\begin{equation}
\Phi(u;y)=\Psi(u)+\frac{1}{2\ep^{2c}}|\GGamma^{-\frac{1}{2}}(y-Ku)|^2,
\label{eq:phi}
\end{equation}
where $\Psi$ is defined in \eqref{eq:psi}. We define the Onsager-Machlup functional $J^\ep$, associated with the measure $\nu^y$ by 
$$
J^\ep(u)=\left\{\begin{array}{rl}
\frac12 \|u\|_{\mathcal{E}}^2+\Phi(u;y) & \text{if } u\in \mathcal{E},\\
\infty & \text{if } u\notin \mathcal{E}.
\end{array}
\right.
$$
Recall $(\mathcal{E},\|\cdot\|_\mathcal{E})$ is the Hilbert space and corresponding norm defined in \cref{ssec:Not}. From a probabilistic perspective, $(\mathcal{E}, \|\cdot\|_\mathcal{E})$ is the Cameron-Martin space associated with the Gaussian measure $\sN(0,C)$ \cite[Definition 6.26]{dashti2013bayesian}. We have the following result demonstrating the role of the Onsager-Machlup functional defined on the Cameron-Martin space from \cite[Theorem 3.5]{DLSV13}.

\begin{proposition}
\label{prop:2}
There exists a  MAP estimator for the posterior measure $\nu^y$ which is a minimizer of the functional $J^\ep$.
\end{proposition}

The functional $J^\ep(u)$ can be written as
\begin{equation}
\label{eq:conc1}
J^\ep(u)=\ep^{-2a_1-3}I^\ep(u),
\end{equation}
where
\begin{align*}
I^\ep(u)=&\frac12{\delta\ep^3}\|\triangle u\|^2_{L^2(\OOmega)}
+\frac12\delta q\ep^{3+2(a_1-a_2)}\|\nabla u\|^2_{L^2(\OOmega)}
+\frac12\delta\tau^2 \ep^{3+2(a_1-a_3)}\|u\|^2_{L^2(\OOmega)}\\
&+r\ep^{3+2a_1-b}\int_\OOmega\frac{1}{4}\bigl(1-u(x)^2\bigr)^2\,\dee x+\frac{1}{2}\ep^{3+2a_1-2c}\bigl|\GGamma^{-\frac{1}{2}}(y-Ku)\bigr|^2.
\end{align*}
We consider the case where
\begin{equation}
\label{eq:scalings2}
a_2-a_1=1, \quad 3+2a_1=b-1=2c,\quad 3+2(a_1-a_3)=a>0.
\end{equation}
With these parameter constraints the functional $I^\ep(u)$ becomes,
for $u\in H^2_\#(\OOmega)$, 
\begin{multline*}
I^\ep(u)=\int_\OOmega\left(\frac12{\delta\ep^3}|\triangle u|^2
+\frac12\delta q\ep|\nabla u|^2
+\frac{r}{4\ep}\bigl(1-u(x)^2\bigr)^2+\delta\tau^2\ep^a u(x)^2\right)\dee x\\
+\frac{1}{2}
\bigl|\GGamma^{-\frac{1}{2}}(y-Ku)\bigr|^2
\end{multline*}
and $I^\ep(u)=+\infty$ when $u\in H\setminus H^2_\#(\OOmega)$.

\begin{definition}
	Define the following two functionals and constant
\[
I_0^\delta=\frac{1}{2}\int_\OOmega P^\delta|\nabla u|\,\dee x+\frac12|\GGamma^{-\frac{1}{2}}(y-Ku)|^2,\ \ if\ u\in BV_{\rm binary}(\OOmega),
\]
\[
e^\delta(U)=\int_{-\infty}^\infty\left(\frac{1}{2}\delta(U''(t))^2+\frac{q}{2}\delta(U'(t))^2+\frac{r}{4}(1-U(t)^2)^2\right)\dee t;
\]
\[
P^\delta=\inf_{U\ {\rm odd}}e^\delta(U).
\]
\end{definition}
Based on the work  of Hilhorst et al \cite{hil}. 
we  have the following theorem for  $\Gamma-$ convergence of the functional $I^\ep$:

\begin{theorem}\label{thm:3}

Then \[
I_0^\delta=\lim_{\ep\to 0} I^\ep
\]
in the sense of $\Gamma-$convergence in the strong $L^1(\OOmega)$ topology.
\end{theorem}
\begin{proof}See \cref{sec:A}. 
\end{proof}

This  theorem shows that the MAP estimator is, for small observational noise
$\ep^c\eta$, close to a perimeter regularization. Furthermore,
since $2a_1+3>0$, \cref{eq:conc1} together with the preceding $\Gamma-$limit theorem 
suggest that, when $\ep \ll 1$, the measure will approximately concentrate
on a single point close to a minimizer of $I_0^\delta.$ 
Our numerical results, presented in the next section, 
support this conjecture.

\begin{remark}
This establishes a link with perimeter regularization for the Bayesian phase-field approach at the level of the MAP estimator. This is not available for the Bayesian level set method because MAP estimators do not exist \cite{igl}. Conversely, the Bayesian level set approach has a link to perimeter length at the level of samples from the posterior (see \cref{lem:per}), which is not true for the phase-field approach. We explain why this is the case.
Recall that almost sure properties of the prior are inherited in the posterior.
Prior samples are drawn from the centred Gaussian with covariance $C$, 
which corresponds to choosing $\alpha=2$. In 
dimension $d \ge 2$ we thus do not have $\alpha>1+d/2$ and 
we cannot deduce that samples have finite perimeter almost surely.  
(Numerical results, reported in subsection \ref{sssec:pp}, 
demonstrate that the lower
bound $\alpha=1+d/2$ is sharp and that random draws beneath this value
do not have finite perimeter almost surely).
MAP estimators, on the other hand, will live in the Cameron-Martin
space of the underlying Gaussian reference measure and are necessarily
smoother than draws from the measure itself \cite{DLSV13}.
Thus there is no contradiction between the fact that 
the MAP estimator, for small $\ep$, approximately
penalizes the perimeter, whilst samples from the posterior may have
infinite perimeter. 
\end{remark}

\begin{remark}
The phase-field approach has the (undesirable) property that, 
for $c>0$, the prior construction
depends on the noise-level $\ep^c$ through \eqref{eq:scalings2} meaning that it is not strictly
Bayesian.
\end{remark}

\section{Numerical Simulations}
\label{sec:N}

\subsection{Test Problems}
\label{ssec:linip}

We test the three inversion techniques on three images referred to as Truth A, Truth B and Truth C. 
These are three fields $u^{\dagger}$ lying in the  set of $BV_{\rm binary}$ images   and are illustrated in 
the obvious way  in \cref{f:truth}.  Truth A and Truth B are observed on a uniform grid of $15\times 15$ points, 
Truth C is observed at 50 uniformly distributed points, and all of these observations are corrupted by additive 
Gaussian noise as in equation \cref{eq:basiceq}.
Pointwise observation does not fit our theory as we assume $K$ is linear on $L^1(D)$; however mollification can be used to address this and leads to results which are not different in any substantive way, as noted in \cite{IglLawStu13}. 
\begin{figure}
\begin{center}
\includegraphics[width=\textwidth,trim=4.5cm 0.5cm 1.5cm 1cm, clip]{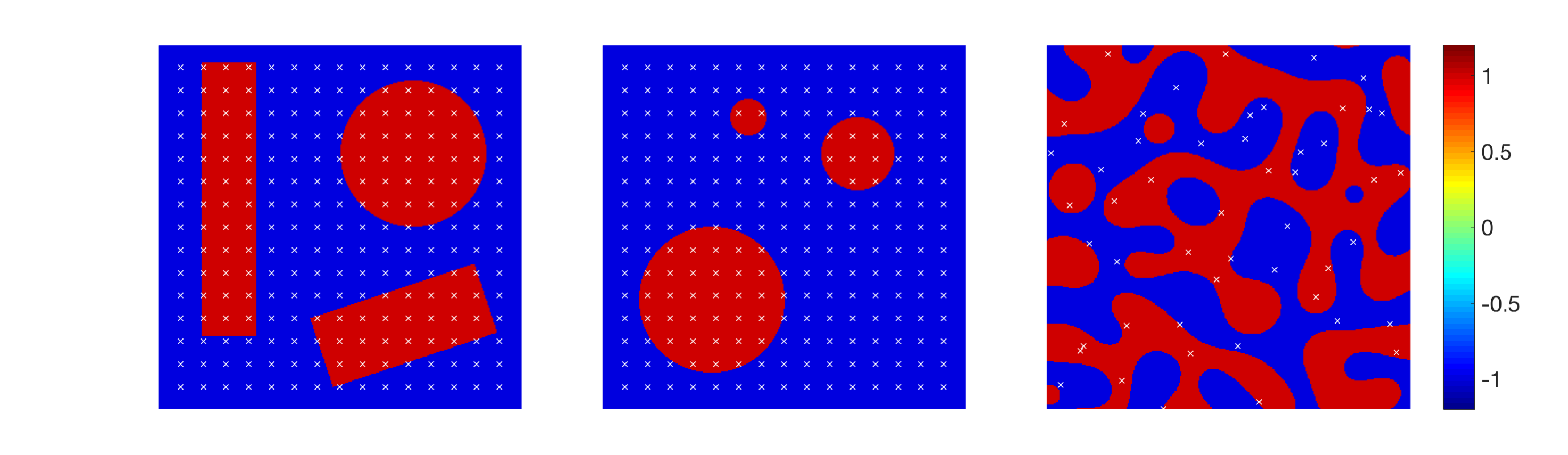}
\end{center}
\caption{The three true fields used for inversion; the field on the left will be referred 
to as Truth A, the field in the middle as 
Truth B and the field on the right as Truth C. The sets of observation points are shown in each figure.}
\label{f:truth}
\end{figure}

In order to avoid an inverse crime \cite{kaipio2006statistical}, Truth A and Truth B are generated on a square 
mesh of $2^{16}$ points, and Truth C is generated on a square mesh of $320^2$ points, but the Gaussian random fields are constructed over a mesh of $2^{14}$ points ($N=2^7$ in \cref{ssec:GF}). We perform numerical experiments in both the small noise and order one noise regimes.

\subsubsection{Small Observational Noise Set-Up}
\label{ssec:thisone}
We set $c=3/2$ and $\ep=0.01.$ The implied standard deviation
of the observational noise is thus $0.001.$ 
We make the 
choices of parameters in the prior covariance 
operator $C(=C_{\ep,\tilde \delta,\tilde a})$, $a_1 = 0$, $a_2 = 1$, $a_3 = 0$, $b=4$ and $r = 1$ for both the Bayesian level set approach 
and the phase-field approach. For the Bayesian level set approach  we set $\delta = 1$, $q=0$, $\tau = 50$ and $\alpha=3$ whereas for the phase-field approach $\delta = 0.01$, $q=0.1$,  $\tau=1$ and $\alpha=2$.
Thus we ensure that  the relations \cref{eq:scalings2} hold so that the phase-field  MAP estimator for $\nu^y$ 
 approximates the minimizer of $I_0^\delta$ as given 
in \cref{thm:3} and we expect the posterior mass to concentrate fairly
close to this MAP estimator.
 Note that in 
general we need not insist on the parameters being related via \cref{eq:scalings2} for 
the level set formulation; this is because, unlike the 
phase-field formulation,
there is no MAP estimator whose properties we are
seeking to control via parameter choices. 
For these small noise experiments the
GP regression used the same parameters as 
for the Bayesian level set method. 
\subsubsection{Order One Observational Noise Set-Up}
We set $c=0.$ Note that now $\ep$ does not enter the observational
noise; it is simply a parameter that enters the prior and so the phase-field formulation is truly Bayesian.  With this choice
of $c$ we require, for the phase-field formulation,
$a_1=-3/2, a_2=-1/2, a_3=-1, b=1$. 
We also set $\delta = 100$, $q=0.1, \tau=1$, $\alpha=2$ 
and $r = 1$. For the level set formulation we retain the same choice of parameters as for the small noise case above.
For the GP regression we
use the same parameters as for the phase-field approach for
these order one noise experiments.

\subsection{Sampling From 
The Gaussian Prior Space}
\label{ssec:GF}

We describe how to sample numerically from Gaussian 
priors $\zeta_0=\sN(0,\sC)$ that are
key to the techniques outlined in the
preceding two subsections. Here $\sC$ is either $C^{\frac{\alpha}{2}}$ or $C$.
We consider the case that $D$ is the  unit square $(0,1)^2$.
Let $\{\lambda_k\}$ denote the eigenvalues of $\sC$ in increasing order with 
corresponding $L^2(D)$-normalized eigenfunctions (which are Fourier
modes) $\{\varphi_k\}$. Then samples $z$ from 
$\zeta_0$ may be expressed through the 
Karhunen-Lo\`eve expansion as
\begin{align}
\label{eq:kl_sum}
z(x) = \sum_{k=1}^\infty \lambda_{k}^{\frac{1}{2}}\xi_{k}\varphi_{k}(x),\quad \xi_k\sim \sN(0,1)\text{ i.i.d.}
\end{align}
We implement an approximation to
this by jointly approximating the field via spectral truncation and evaluation on a discrete grid of points. Such an approximation may 
be efficiently implemented using the Fast Fourier Transform. We work on
a uniformly spaced grid $\{x_i\}$ of $N^d$ points in $D$. An approximate sample on this grid is then given by 
\[
z^N(x_i) = \sum_{k=1}^{N^d} \lambda_{k}^{\frac{1}{2}}\xi_{k}\varphi_{k}(x_i), \quad \xi_k\sim \sN(0,1)\text{ i.i.d.}
\]
All of our numerical results  are performed on the
two dimensional grid which arises from this approach 
to generating Gaussian random fields in dimension $d=2.$
In practice we choose $N=2^7$, and so the discrete grid 
for our inversion is $2^{14}$ points.

\subsection{MCMC Simulation}
Markov Chain Monte Carlo (MCMC) simulations may be used to sample measures $\nu^y$. 

 In all MCMC runs we generate $10^6$ samples and, when computing means, discard the first 
$5\times 10^5$ samples as burn-in . 

We employ the preconditioned Crank-Nicolson (pCN) algorithm
\cite{beskos2008mcmc,cotter2013mcmc} which may be used to 
sample any measure $\sigma$ of the form
\[
\frac{\dee\sigma}{\dee \sigma_0}(z) = \frac{1}{Z}\exp\left(-A(z)\right),\quad \sigma_0 = \sN(0,C),
\]
without computing derivatives of $A(\cdot).$ 
 Both of our posterior measures can be written in this form. For the Bayesian level set approach we have
\[
\frac{\dee\nu^y}{\dee\mu_{0,\alpha}}=\frac{1}{Z}
\exp\left(-\frac{1}{2\ep^{2c}}\Bigl|\GGamma^{-\frac{1}{2}}\bigl(y-KS(v)\bigr)\Bigr|^2\right),
\]
whereas  for the phase-field approach we have
\[
\frac{\dee\nu^y}{\dee\mu_0}=\frac{1}{Z}\exp\left(-\frac{r}{\ep^b}\int_\OOmega\frac{1}{4}\bigl(1-u(x)^2\bigr)^2\,\dee x 
-\frac{1}{2\ep^{2c}}\Bigl|\GGamma^{-\frac{1}{2}}\bigl(y-Ku\bigr)\Bigr|^2
\right),
\]
where $\mu_0,\mu_{0,\alpha}$ are Gaussian measures, and with the appropriate normalization constants $Z$.
 The pCN method has the advantage that, unlike the standard Random Walk Metropolis MCMC algorithm, its rate of convergence to equilibrium can be bounded below independently of the number of terms used in the truncated Karhunen-Loève
expansion described in \cref{ssec:GF} \cite{HSV14}.

In the notation of \cite{cotter2013mcmc} for the pCN algorithm, $\beta \in (0,1]$ denotes the proposal variance parameter.  Note that
larger $\beta$ tends to lead to smaller acceptance probability, but
to greater exploration of state space when steps are accepted; the 
optimal $\beta$ is a trade-off between these two competing effects.
Depending on the noise model 
and number of observations, we take the proposal standard deviation 
parameter $\beta$ between $0.02$ and $0.1$ for level set 
simulations and $\beta$ between 
$0.002$ and $0.02$ for phase-field simulations. These choices are made in order to balance acceptance rate and size of proposed move with a view
to optimizing the convergence rate of the Markov chain.

We simply {\em assume} that the resulting Markov 
chains $\{\eta^{(m)}\}$ are ergodic and 
make the approximation for an associated measure $\nu^y$ that 
$${\mathbb E}^{\nu^y} g(\eta) \approx \frac{1}{M}\sum_{m=1}^M g(\eta^{(m)})+e_M$$
where the error $e_M$ is Gaussian with variance $c_g/M$, and $c_g$ is the
integrated auto-correlation of $g(\eta^{(m)}).$ 
We do not impose specific stopping criteria on the Markov chains, rather
we will examine the approximation qualities derived from the chains,
as a function of $M$, and study the convergence to equilibrium of quantities
of interest; in particular in \cref{sssec:comp} we compare the
acceptance probability of the chain, as a function of $M$, for
the level set and phase-field approaches. 
The samples up to step $M$ can then be used to produce point estimates 
for the fields, by calculating, for example, their mean or the sign of their mean. We compare the 
cost of sampling versus the quality of reconstruction with these point estimates, 
for differing  formulations.

\begin{remark}
	The theory in \cite{HSV14} demonstrates ergodicity for problems similar
to those arising in the phase-field formulation. Developing an
analogous theory for the level set formulation is an open and interesting research direction; however our 
numerics do suggest that ergodicity holds in this case too. \end{remark}

\begin{remark}
Preliminary numerical results for the one dimensional analogue of
the problem studied here may be found in \cite{Sivak}. They are
consistent with what we report here in dimension two.
\end{remark}

\subsubsection{Computational Cost}\label{sssec:comp}

For MCMC sampling, which we use for both the phase-field and level set
approaches,  every set of the Markov chain requires an evaluation of $A(u).$ 
Due to the presence of an extra integral term, this evaluation will typically be more expensive 
for the phase-field model than the level set model; for the simulations performed here, 
evaluation of $A(u)$ is approximately twice as expensive for the phase-field model than for the 
level set model. For the GP regression simulations no sampling is required and so the 
computational cost is significantly cheaper; the means were calculated from the expression 
in \cref{sss:gp}, with the cost arising from the matrix multiplications and inversion involved.

For the phase-field and level set approaches, the most significant discrepancy
in computational cost arises from the statistical properties of the 
Markov chain used to sample the posterior approximately. In \cref{f:accept} we show the evolution 
of the local acceptance rates of proposed states for Truth A with small observational noise, as a function of $M$.
The evolutions are similar for the other datasets and so are not presented for 
brevity.The parameter $\beta$, the
proposal variance, is chosen so that the acceptance probability
is neither close to one nor zero; recall that this results
in an order of magnitude smaller value of $\beta$ for the phase-field
method in comparison with level set, meaning that the former method makes
a much slower exploration of the posterior distribution.
\cref{f:accept} suggests that the phase-field chains have not reached equilibrium until after at least $5\times 10^5$ samples, whereas the level set chains converge much earlier. This is illustrated in \cref{f:samples}, which shows a selection of samples for $M=\mathcal{O}(10^4)$ for Truth B with small observational noise. With $M=10^4$ samples, the three inclusions have already been identified by the level set chain, however after $M=3 \times 10^4$
samples the phase-field chain has only started to identify a second inclusion. Thus, even though for both models we produced the same number of samples, it would have sufficed to terminate the level set chains much earlier, significantly reducing the computational cost.

The fact
that the acceptance rates for the phase-field chains are lower than those for the level set chains, despite the proposal standard deviation parameter $\beta$ being one tenth of the size can be understood as
follows. Note that the measure $\nu^y$ can 
informally be thought of as having Lebesgue density 
proportional to $\exp(-J^\ep(u)) = \exp(-\ep^{-2a_1-3}I^{\ep}(u))$. Thus for small $\ep$ the probability mass is concentrated in a small neighborhood of critical points of 
$I^\ep \approx I^{0}_{\delta}$. The MCMC simulations for $\nu^y$ could hence be viewed as a 
form of derivative-free optimization for the functional $J^\ep$.

\begin{figure}
\begin{center}
\includegraphics[width=\textwidth,trim=1cm 0cm 0cm 1cm, clip]{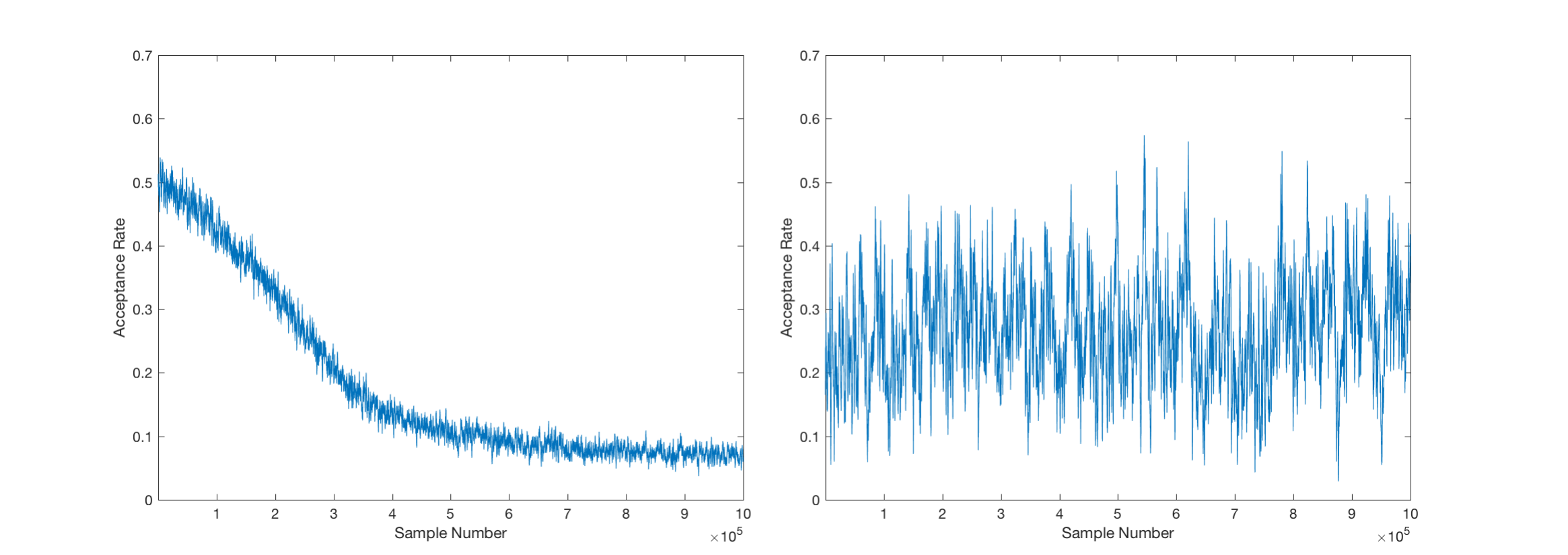}
\end{center}
\caption{The evolution of the acceptance rate of proposals for the phase-field (left) and level 
set (right) MCMC chains, for Truth A with small observational noise. 
Acceptance rates are calculated over a moving window of $1000$ samples.}
\label{f:accept}
\end{figure}

\begin{figure}
\begin{center}
\includegraphics[width=\textwidth,trim=2cm 0cm 0cm 1cm, clip]{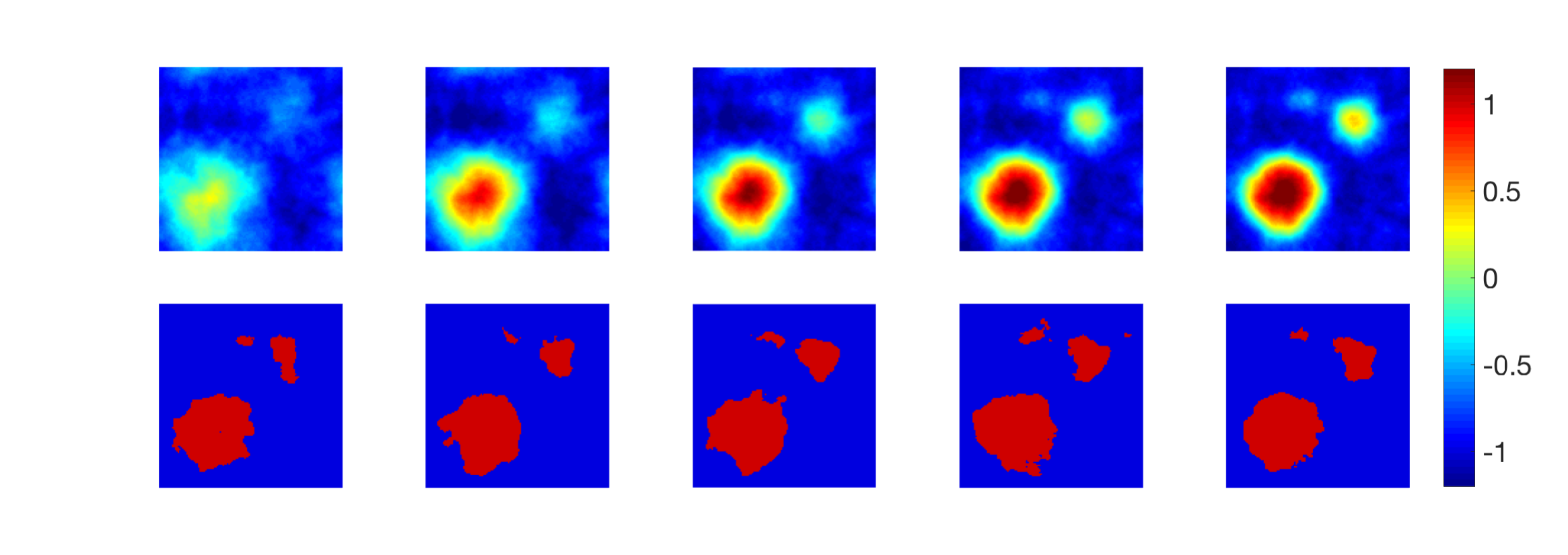}
\end{center}
\caption{Examples of samples near the start of chains for Truth B with small observational noise. 
Sample numbers 10000, 20000, 30000, 40000 and 50000 are shown from left-to-right for the 
phase-field chain (top) and the level set chain (bottom).}
\label{f:samples}
\end{figure}

\subsection{Reconstruction Of The Means}

\subsubsection{Small Observational Noise}
\label{ssec:smallnoise}
In \cref{f:means32} we present sample means associated with small-noise observations for 
the phase-field, level set and GP regression models, both with and without thresholding by $S$. 
Note that the phase-field and GP regression models attempt to fit the un-thresholded field to the 
data points, whereas the level set method attempts to fit the thresholded field; the un-thresholded field 
for the level set method is hence on a different scale to the other two models. 

For Truth A and Truth B, the general quality of the reconstruction is similar for 
all three models after thresholding, though the level set method does not overfit to the 
datapoints as significantly as the other two methods; this
overfitting for the phase-field and GP regression is
manifest in a boundary for the largest inclusion in Truth B which has variations
on the scale of the observational noise. Another noticeable effect
in the quality of the phase-field and GP regression, manifest
in Truth A, is that the edges of the circular inclusion are rendered
approximately piecewise linear; this might be ameliorated by using
a small mesh increment to $\ep$ ratio. The level set method has
no small length scale to resolve, and hence does not suffer from this effect.

For all three models reconstruction of Truth C is fairly inaccurate 
as the sparse observation network does not resolve the length scale on
which the true field varies. The level set and GP regression models perform similarly, whereas 
the phase-field model places much more mass into the positive class; it is likely that
this reflects a lack of convergence of the Markov chain for the phase-field
model.

\begin{figure}
\begin{center}
\includegraphics[width=.9\textwidth,trim=4cm 0.5cm 3.5cm 1cm, clip]{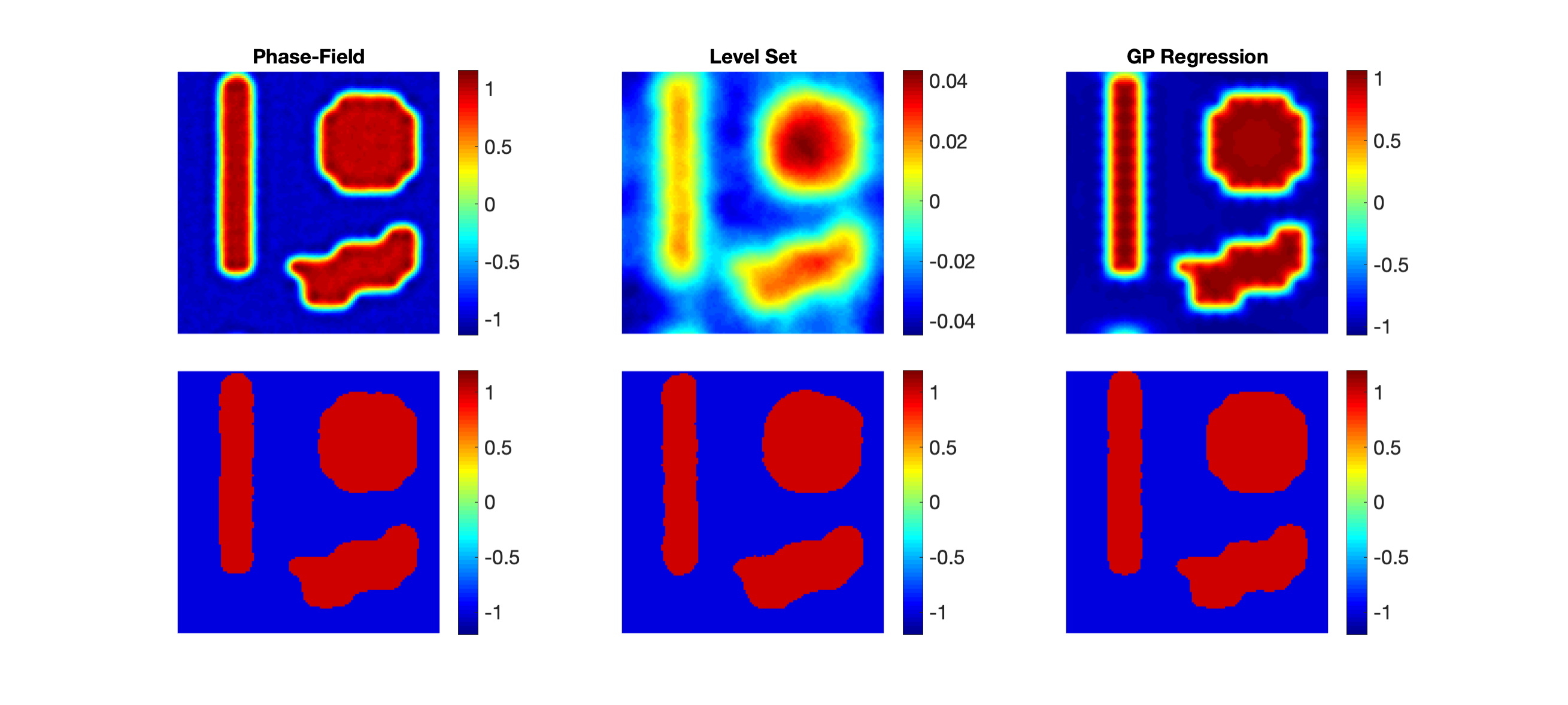}\\
\includegraphics[width=.9\textwidth,trim=4cm 0.5cm 3.5cm 1cm, clip]{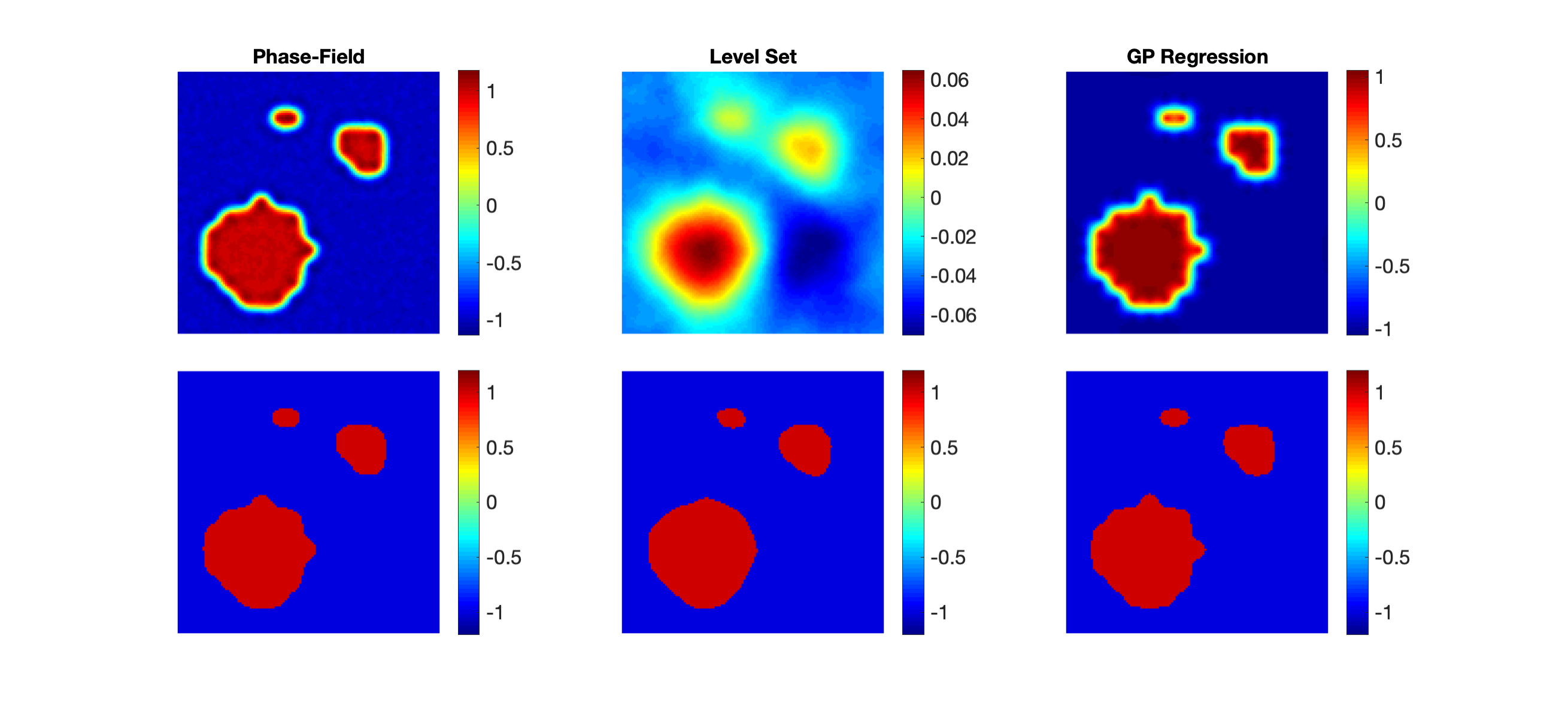}\\
\includegraphics[width=.9\textwidth,trim=4cm 0.5cm 3.5cm 1cm, clip]{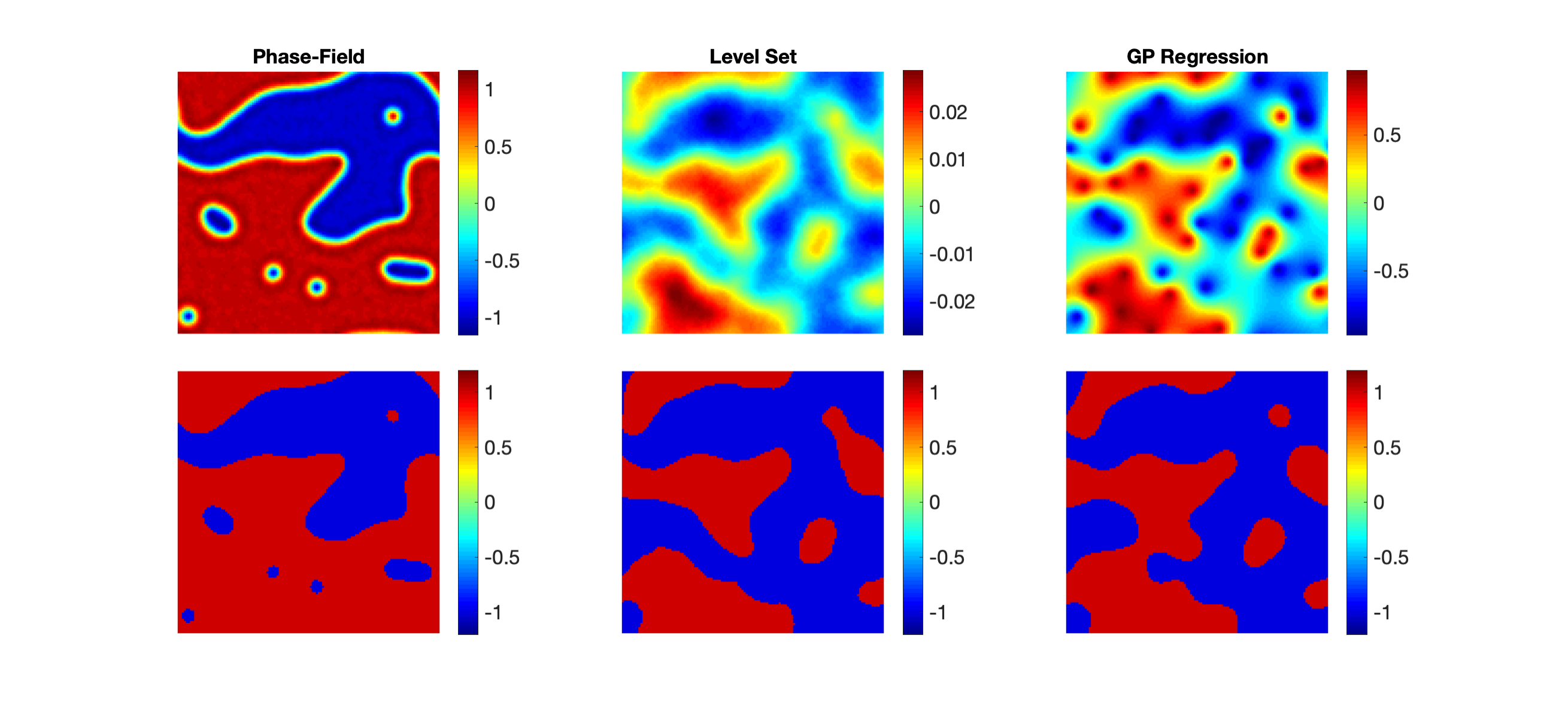}
\end{center}
\caption{Sample means for Truth A (top block), Truth B (middle block) and Truth C (bottom block) with 
small observational noise. To top row of each block shows Monte Carlo approximations 
to $\mathbb{E}^{\nu^y}(v)$, the underlying continuous fields, and the bottom row in each block 
shows Monte Carlo approximations to $S\big(\mathbb{E}^{\nu^y}(v)\big)$, the thresholded fields.}
\label{f:means32}
\end{figure}

\begin{figure}
\begin{center}
\includegraphics[width=.9\textwidth,trim=4cm 0.5cm 3.5cm 1cm, clip]{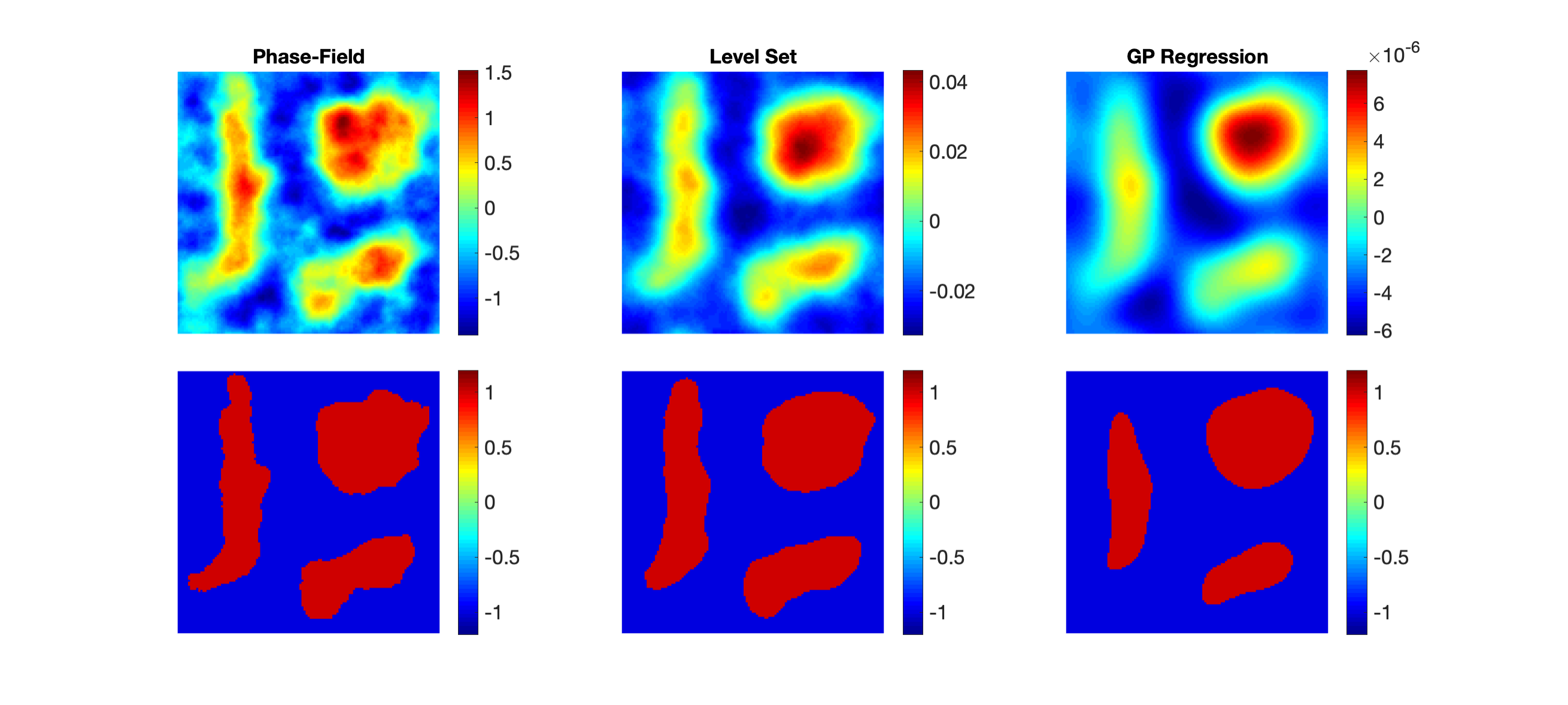}\\
\includegraphics[width=.9\textwidth,trim=4cm 0.5cm 3.5cm 1cm, clip]{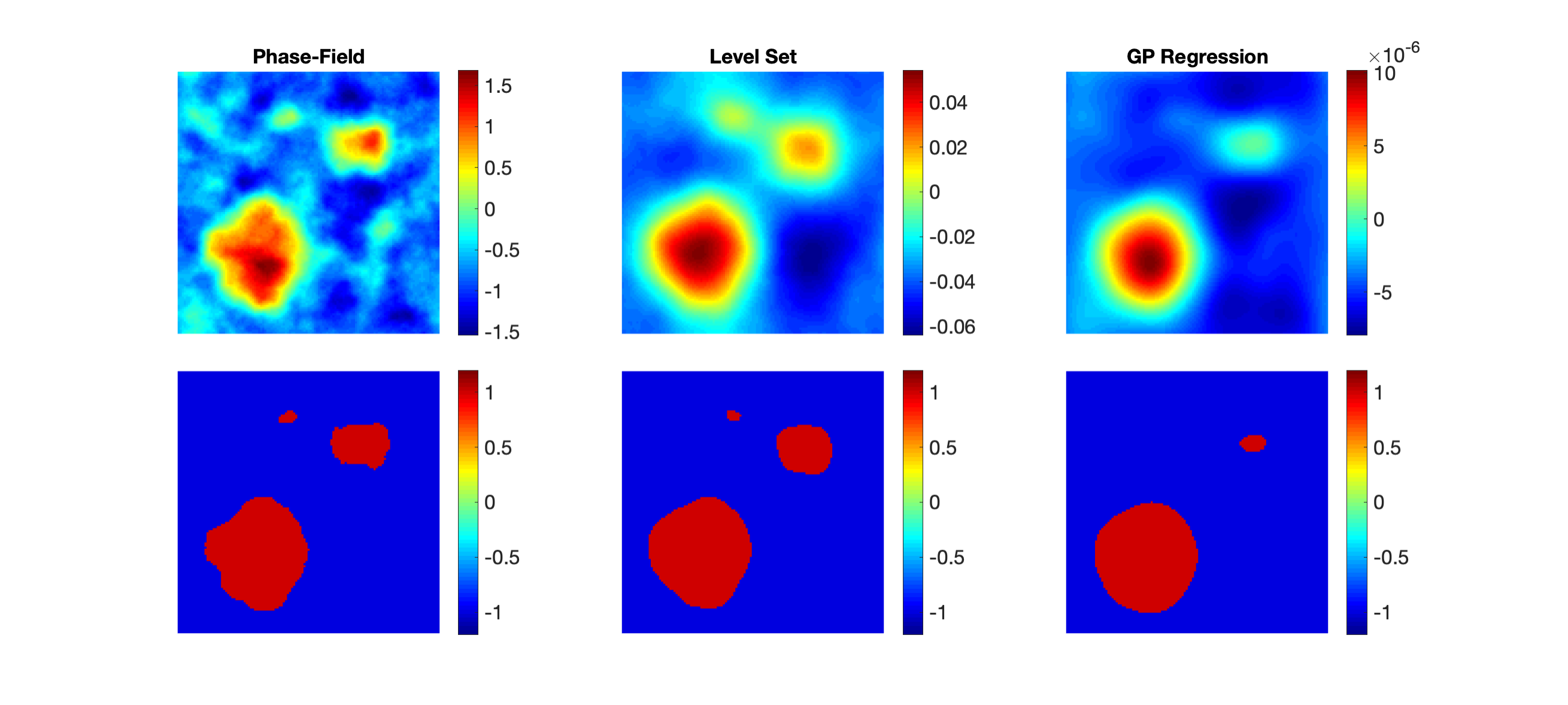}\\
\includegraphics[width=.9\textwidth,trim=4cm 0.5cm 3.5cm 1cm, clip]{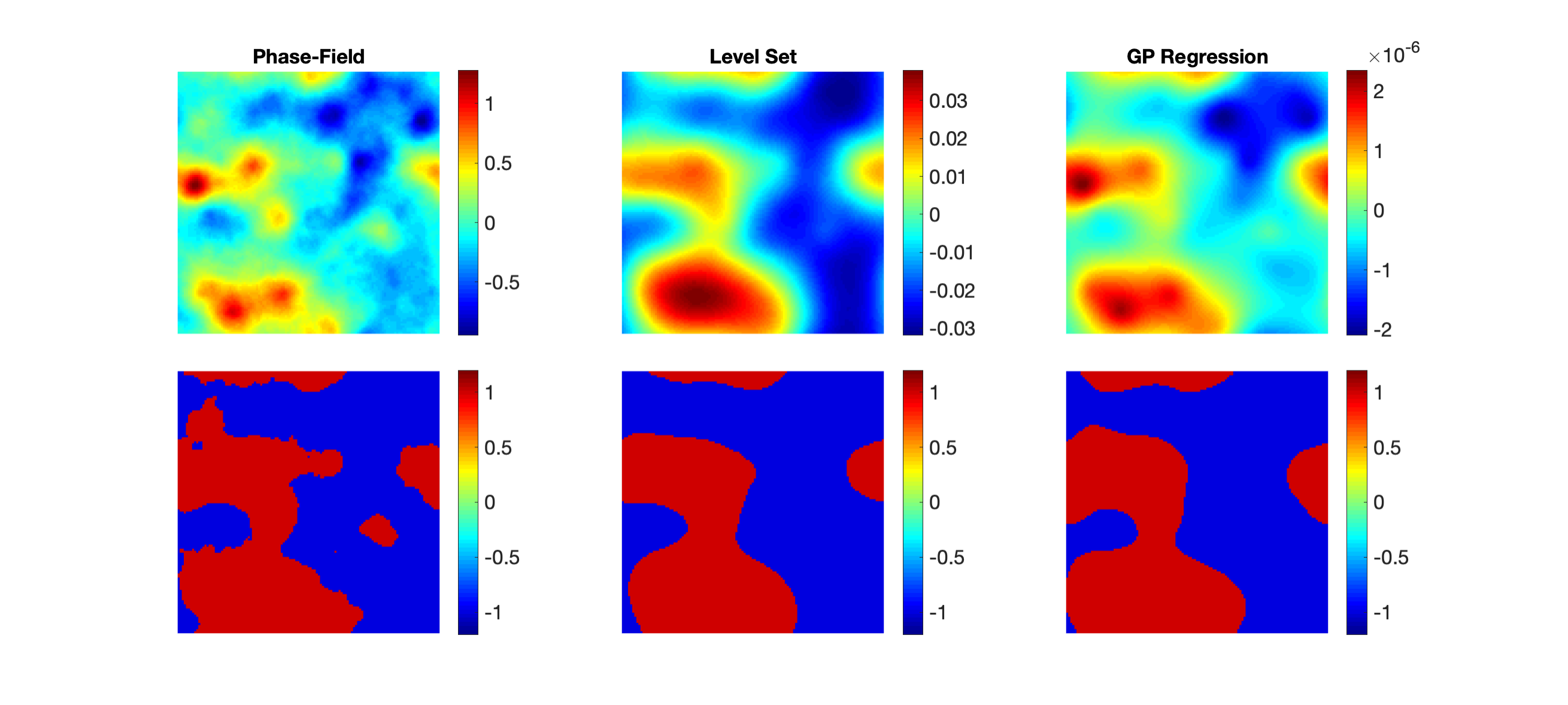}
\end{center}
\caption{Sample means for Truth A (top block), Truth B (middle block) and Truth C (bottom block) 
with order one observational noise. To top row of each block shows Monte Carlo approximations 
to $\mathbb{E}^{\nu^y}(v)$, the underlying continuous fields, and the bottom row in each block 
shows Monte Carlo approximations to $S\big(\mathbb{E}^{\nu^y}(v)\big)$, the thresholded fields.}
\label{f:means0}
\end{figure}

\subsubsection{Order One Noise Reconstructions}

In \cref{f:means0} the sample means associated with order one observational noise are shown. 
As would be expected, reconstruction quality is generally poorer than for the small-noise 
observations, though overfitting to the observational 
noise is no longer an issue for the phase-field and GP 
approaches. The three models perform similarly, though  there seems to be an increased amount of 
penalization on the length of the interface from left-to-right. Without thresholding, 
the GP regression means provide poor estimates of
the truth in terms of scale, due to the far weaker influence of the likelihood and 
lack of prior information enforcing values close to $\pm 1$.

\subsection{Perimeter Learning}
\label{sssec:pp}
Here we study perimeter learning for the Bayesian level set method. The length of the zero level set of $\eta^N$ may then be approximated by using the discrete variation
of $w^N:=S(\eta^N)$,
\[
\ell(N) = \frac{1}{2N^2}\sum_{i,j=1}^N {|\mathsf{D}^N w^N(x_i,y_j)|} \approx \frac{1}{2}\int_D |\nabla w^N(x,y)|\,
\]
where the operator $\mathsf{D}^N$ approximates the gradient on the grid $\{x_i,y_j\}$ via central differences.
Using this we investigate numerically
the length of the level sets of these samples. We have shown in
\cref{lem:per} that choosing $\alpha>1+d/2$ is sufficient
to ensure almost sure finite length of level sets. Numerical
experiments   indicate that this is a sharp result. In 
 \cref{f:interfaces} interface lengths for prior distributions approximated as described above for a 
 single realization of  $\{\xi_k\}$ in \cref{eq:kl_sum} are displayed as a function of $N$ for varying values of $\alpha$.   
  We use $d=2$  and observe that for $\alpha < 2$ the length of the interface diverges algebraically with $N$ (left hand panel), 
for $\alpha = 2$ it diverges logarithmically (right hand panel shows 
this best), and for $\alpha > 2$ it converges to a constant (both left
and right hand panels show this). 
The results, then, suggest that level sets have finite length almost surely
if and only if $\alpha>1+d/2.$

\begin{figure}
\begin{center}
\includegraphics[width=0.49\textwidth,trim=1cm 0cm 0cm 1cm, clip]{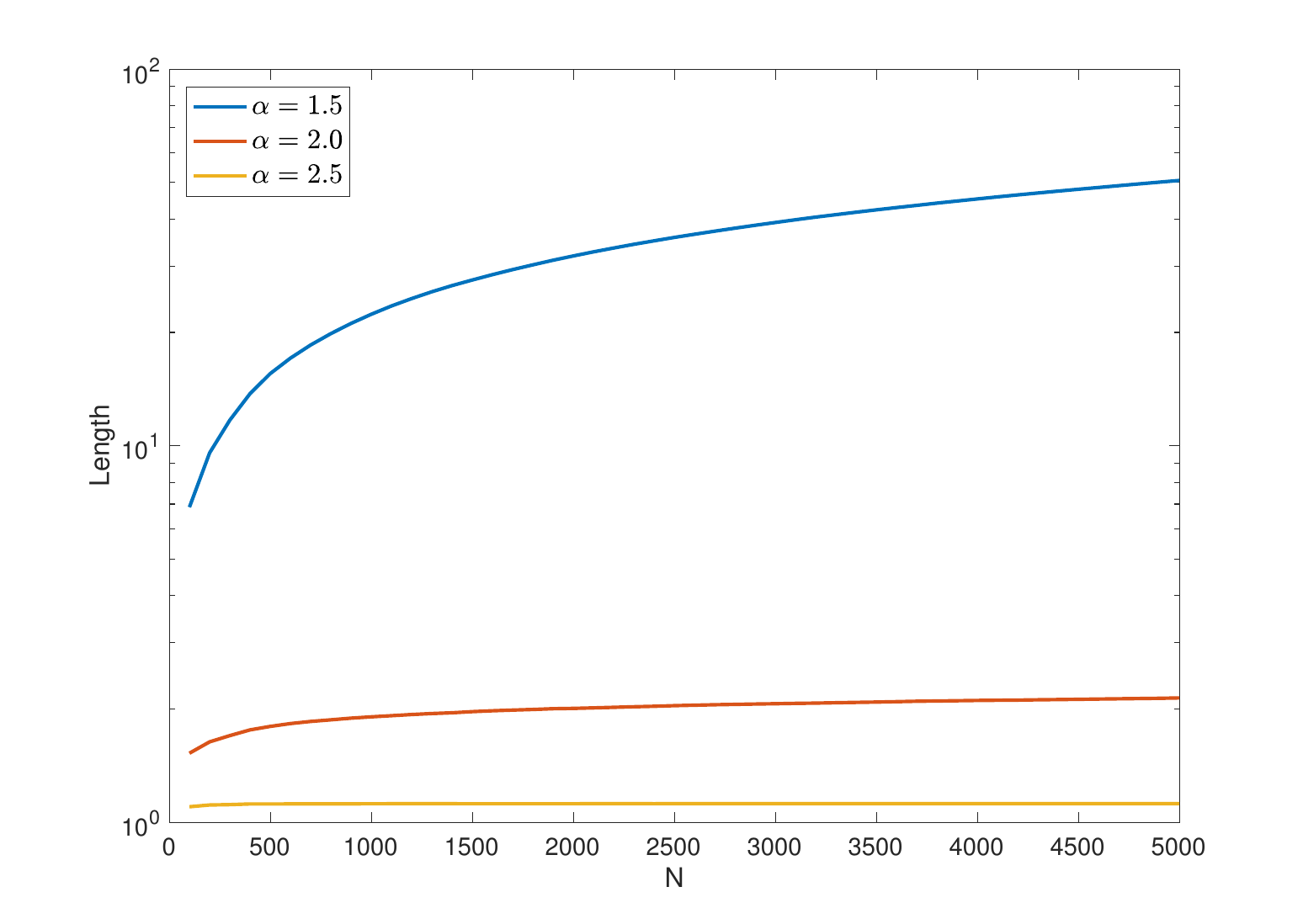}
\includegraphics[width=0.49\textwidth,trim=1cm 0cm 0cm 1cm, clip]{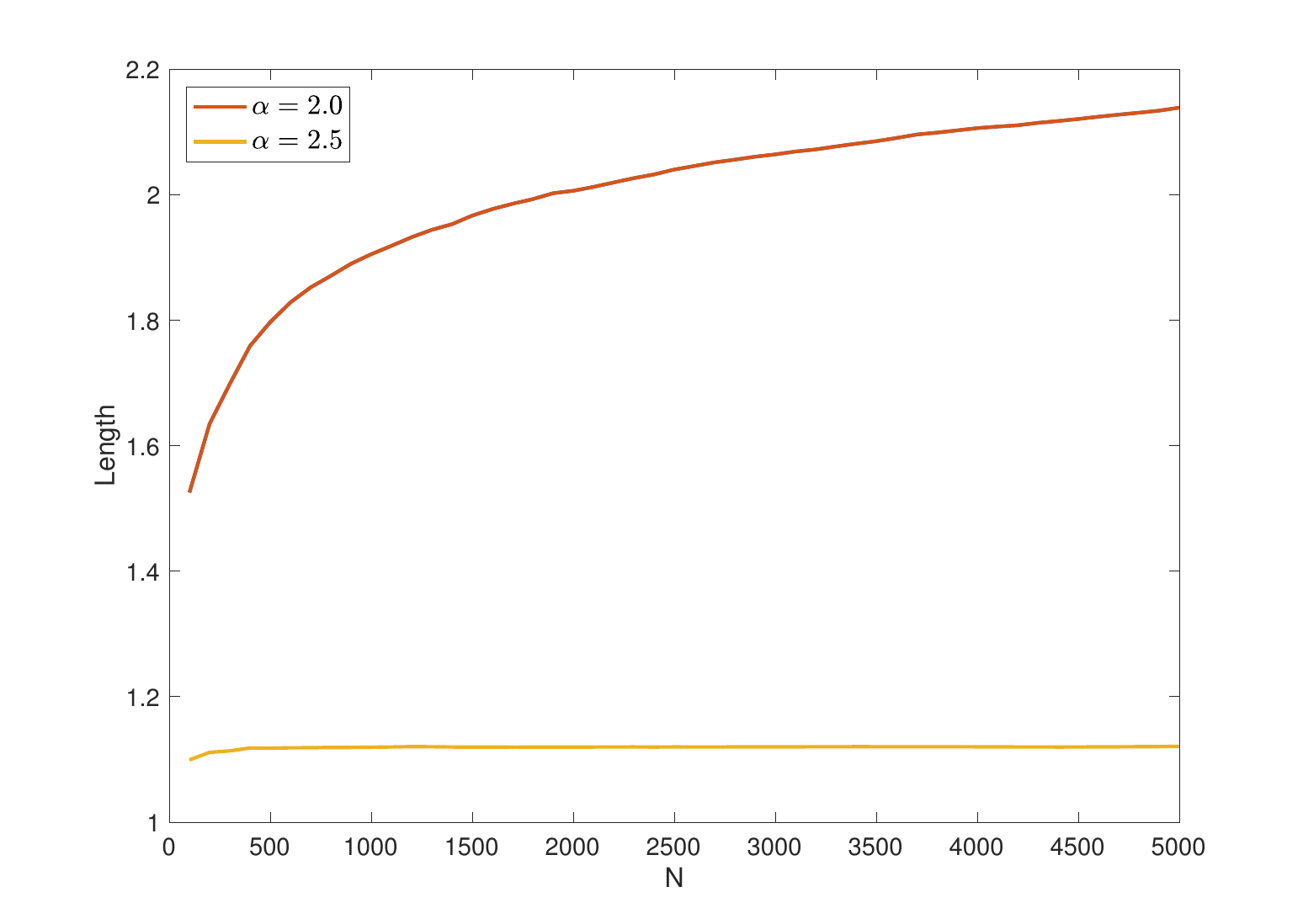}
\caption{The dependence of the length of the zero level set of a Gaussian
sample, as a function of numerical approximation level $N$ and of 
prior regularity parameter $\alpha$. (Left) logarithmic axis, (right) linear axis.}
\end{center}
\label{f:interfaces}
\end{figure}

In order to compare the perimeter distribution between the prior and posterior, 
the choice $\alpha > 1+d/2$  ensures that  the
length of the zero level set 
is well-defined so in two   dimensions we take   $\alpha=3$. The results for recovery of Truth B
are shown in \cref{f:interfaces_dist}.
Whilst the perimeter still retains some variation under the posterior, the variation
is much lower and, in contrast to the prior, it is concentrated 
around the true value of the perimeter. We see that though the Bayesian level set approach does not explicitly penalize the perimeter, it has the ability to estimate the perimeter, and quantify uncertainty in the estimation.

\begin{figure}
\begin{center}
\includegraphics[width=0.6\textwidth,trim=1cm 0cm 0cm 0cm, clip]{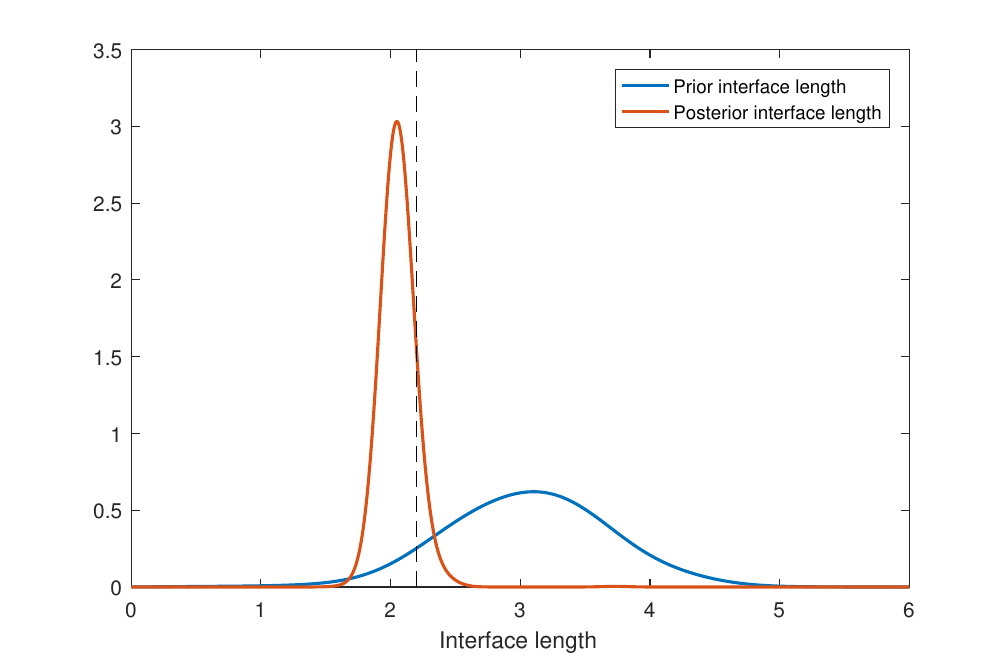}
\caption{The distribution of the perimeter under the prior and posterior 
distribution. The vertical dashed line indicates the perimeter of the true field.}
\end{center}
\label{f:interfaces_dist}
\end{figure}

\section{Eikonal Equation} \label{sec:eik}
In this section we build on what we have learned so far for the
linear inverse problem defined by \eqref{eq:basiceq} and use it to  
study a nonlinear inverse problem from the eikonal equation.
Gaussian process regression is fast to implement, and appears to give
qualitatively comparable accuracy to the Bayesian level set method; but
it does not generalize to nonlinear problems and so we do not consider it 
further. The experiments in the previous section, set-up in \cref{ssec:linip},
suggest that the Bayesian level set approach to binary recovery 
has two advantages over the Bayesian phase-field formulation, 
for the linear inverse problem considered: the level set method is faster
and draws from the posterior contain information about the true
perimeter. 
Thus we focus attention purely on the Bayesian level set method. Within
this context we also demonstrate the benefits of hierarchical Bayesian
inversion.

\subsection{The Forward Equation} \label{sec:eikfp}
Let $x_0\in D$ be a wave emitting source, and define the first arrival time of the wave at $x\in D$ as $T(x)$. 
The wave passes through a medium which adjusts the wave speed according to a
scalar function $u\colon D\to \bbR_+$ known as the \emph{slowness}.
It is shown in \cite{book:Lio82,Son86} that $T(x)$ may be viewed as solution
of a stationary Hamilton-Jacobi equation, namely the following 
eikonal equation:
\begin{eqnarray}
 |\nabla T (x)| = u(x),& \qquad&\forall x \in D \setminus \{x_0\},\label{eq:eik}\\
 T(x_0)=0,& &\label{eq:ptconstr}\\
 \nabla T(y) \cdot n(y) \geq 0,& & \forall y \in \p D,\label{eq:sonerbc}
\end{eqnarray}
where $n$ is the outward pointing unit normal. The Soner boundary condition \eqref{eq:sonerbc} ensures ray paths terminate at $\p D$ \cite{Son86}. The recovery of the slowness function $u$ from observations of arrival times is known as first arrival traveltime tomography. Extensive discussion of the well-posedness of the forward problem can be found in \cite{DecEll04,DecEllSty11,DunEll19}. For our application, we consider a binary slowness function $u\colon D \to \{u_{\min},u_{\max}\}$ where $0<u_{\min}\leq u_{\max}$.

We define the solution map $\mmm[G]$ mapping the slowness $u$ to the 
travel times $T$ via solution of the eikonal equation with source $x_0$.
Because we are interested in binary slowness functions we also introduce
\begin{equation}
\mmm[S](v)=S(v)\cdot(u_{\max} - u_{\min})/2 +(u_{\max} + u_{\min})/2;
\end{equation}
here $S(v)$, is the sign function defined in \cref{sec:B}.

To solve the forward problem we first discretize using an upwind finite difference scheme. We then use a fast marching procedure (see \cite{Set99}) to solve the discrete eikonal equation. A formulation of the discretization and marching algorithm, along with a  proof of numerical convergence is found in \cite{DecEllSty11}.

\subsection{The Bayesian Inverse Problem} \label{sec:eikip}
Let $\eta\sim \sN(0,\Sigma)$ be a normal random variable in $\bbR^J$. Fix $x_0\in \bar D$.
Defining the observation map $K$ taking traveltimes to $\bbR^J$ we
define the inverse problem of finding $v$, given observed data $y$
satisfying 
\begin{equation}\label{eq:absinv}
 y = K\circ \mmm[G]\circ \mmm[S](v) +\ep^c\eta = K\circ \mmm[G](u) + \ep^c\eta.
\end{equation}
We will assume that $K$ is defined so that
the data $y$ is a set of observed first hitting times at fixed known receiver locations $\{z_j\}_{j=1}^J \in \bar D$. 
The random variable 
$y|v \sim \sN(K\circ\mmm[G]\circ \mmm[S](v),\ep^{2c}\Sigma)$, 
leading to negative log likelihood defined, up to an additive constant, by
\[
\Phi(v;y)= \frac{1}{2\ep^{2c}}\left| \Sigma^{-\frac{1}{2}}\big(y-K\circ\mmm[G]\circ \mmm[S](v)\big)\right|^2.
\]
We will treat problems of multiple sources $\{x_0^m\}_{m=0}^M$ as multiple experiments, with solution maps $\mmm[G]^m$ each producing data $y^m\in \bbR^J$. The natural extension is to consider the following negative log likelihood,
\begin{equation}
\label{eq:needl}
\Phi(v;y) = \frac{1}{2\ep^{2c}}\sum_{m=1}^M\left| \Sigma^{-\frac{1}{2}}\big(y^m-K\circ\mmm[G]^m\circ \mmm[S](v)\big)\right|^2, 
\end{equation}
for the $m$ observations $y = (y^1,\dots,y^m).$
We assume that we are given prior measure
$\zeta_0=\sN\bigl(0,\sC\bigr)$ and let $\zeta^y(\dee v)$
denote the probability distribution
of the conditioned random variable $v|y$. Then using Bayes' theorem (see \cref{prop:12}), we deduce that
$\zeta^y$ is a probability measure supported on continuous functions,
determined by
\[
\frac{\dd \zeta^{y}}{\dd \zeta_0} = \frac{1}{Z}\exp\Big(-\frac{1}{2\ep^{2c}}\sum_{m=1}^M\Big| \Sigma^{-\frac{1}{2}}\big(y^m-K\circ\mmm[G]^m\circ \mmm[S](v)\big)\Big|^2\Big),
\]
with normalization constant $Z$.
This problem has been formulated for piecewise constant slowness function with Ginzburg-Landau type regularization in the deterministic setting \cite{DunEll19}, where simulations and proofs of convergence of numerical schemes can be found. We choose here to use a level set formulation for the reasons discussed
at the start of the section.

\subsubsection{Hierarchical Inference For Inverse Lengthscale} \label{ssec:hierbayes}

In hierarchical\linebreak Bayesian inference we employ
a Gaussian prior measure $\zeta_0=\sN(0,\sC(\tau))$ in which
the covariance $\sC$ depends on parameter $\tau>0$ which we
interpret as an additional unknown to be learned during the 
inversion process. In particular we will work with settings
in which $\tau$ has interpretation as an inverse length-scale.
To this end, define the hierarchical prior $\zeta_0$ by decomposing as follows:
\[
 \zeta_0(\dee v,\dee \tau) = \zeta_0(\dee v|\tau) \pi_0(\tau) \dd \tau.
\]
We call $\pi_0(\tau)$ the hyperprior. Generalizing the 
derivation of the posterior in the preceding subsection,
we now find that the distribution
of $v,\tau|y$ is determined by probability measure $\zeta^y(\dee v, \dee \tau)$
defined by 
\begin{equation}\label{eq:posthyp0}
 \zeta^y(\dee v, \dee \tau) = \frac{1}{Z} \exp(-\Phi(v;y))\zeta_0(\dee v|\tau) \pi_0 (\tau)\dd \tau,
\end{equation}
for normalization constant $Z$.
For reasons discussed in \cite{papaspiliopoulos2007general,yu2011center} 
it can be advantageous to reparameterise the hierarchical inverse problem 
in terms of $(\xi,\tau)$, rather than $(v,\tau)$, where $\xi$ is a 
Gaussian white noise distributed as $\sN(0,I),$ so that $v=\sqrt{\sC(\tau)}\xi;$
the underlying latent Gaussian white noise $\xi$ may be identified with
the collection of i.i.d. unit Gaussians $\{\xi_k\}$ used
to construct prior samples in \cref{ssec:GF}.
Abusing notation we may write the posterior distribution $\zeta^y$, now
for the variable $(\xi,\tau)$, as
\begin{equation}\label{eq:posthyp}
 \zeta^y (\dee \xi,\dee \tau) = \frac{1}{Z} \exp(-\Phi(\sqrt{\sC(\tau)}\xi;y))\zeta_0(\dee \xi) \pi_0 (\tau)\dd \tau.
\end{equation}
Working with variables $(v,\tau)$ as in \eqref{eq:posthyp0} is refered to
as the \emph{centred} problem formulation; using variables $(\xi,\tau)$ 
as in \eqref{eq:posthyp} is refered to as the \emph{non-centred} 
problem formulation.

Numerical evidence described in \cite{DunIglStu17,CheDunPapStu19_preprint}
demonstrates that for level-set based inverse problems use of the
non-centered formulation  in \eqref{eq:posthyp} confers considerable
advantages in terms of speed of convergence of MCMC. We thus employ
the non-centered formulation.
To sample from this distribution we use \cite[Algorithm 6.1]{CheDunPapStu19_preprint}, known as the non-centred pCN-within-Gibbs. This algorithm updates $\xi$ and $\tau$ in separate substeps of a pCN sampling method, and we perform this in a random order each iteration. 

\subsubsection{Hierarchical Inference For Contrast}
We also consider a model hyperparameter that originates in the application (whereas $\tau$ appears when regularizing through the prior). We investigate the contrast between the binary materials, and rewriting the relationship between $u$ and $v$, we obtain
\[
\mmm[S](v)=\frac{\kappa}{2} (S(v)+1) + u_{\min},
\]
where $\kappa=(u_{\max} - u_{\min})>0$ is the contrast. Knowing the parameter $u_{\min}$, we consider the parameter $\kappa$ as a positive random variable to be learnt from the data. This increases flexibility of techniques in application as we can apply our methods to scenarios where slowness contrast is uncertain. We may include this in a non-centred pCN-within-Gibbs algorithm as above, 
by exploiting the form 
$$\mmm[S](v)=\frac{\kappa}{2} (S(\sqrt{\sC(\tau)}\xi)+1) + u_{\min},$$ 
substituting this into the likelihood \eqref{eq:needl}. 
We update each of $\xi,\tau,\kappa$ separately and in random order 
in each iteration.

\subsubsection{Numerical Results: Lengthscale Hyperparameter Only} 

\begin{figure}
\begin{center}
\includegraphics[width=0.48\textwidth]{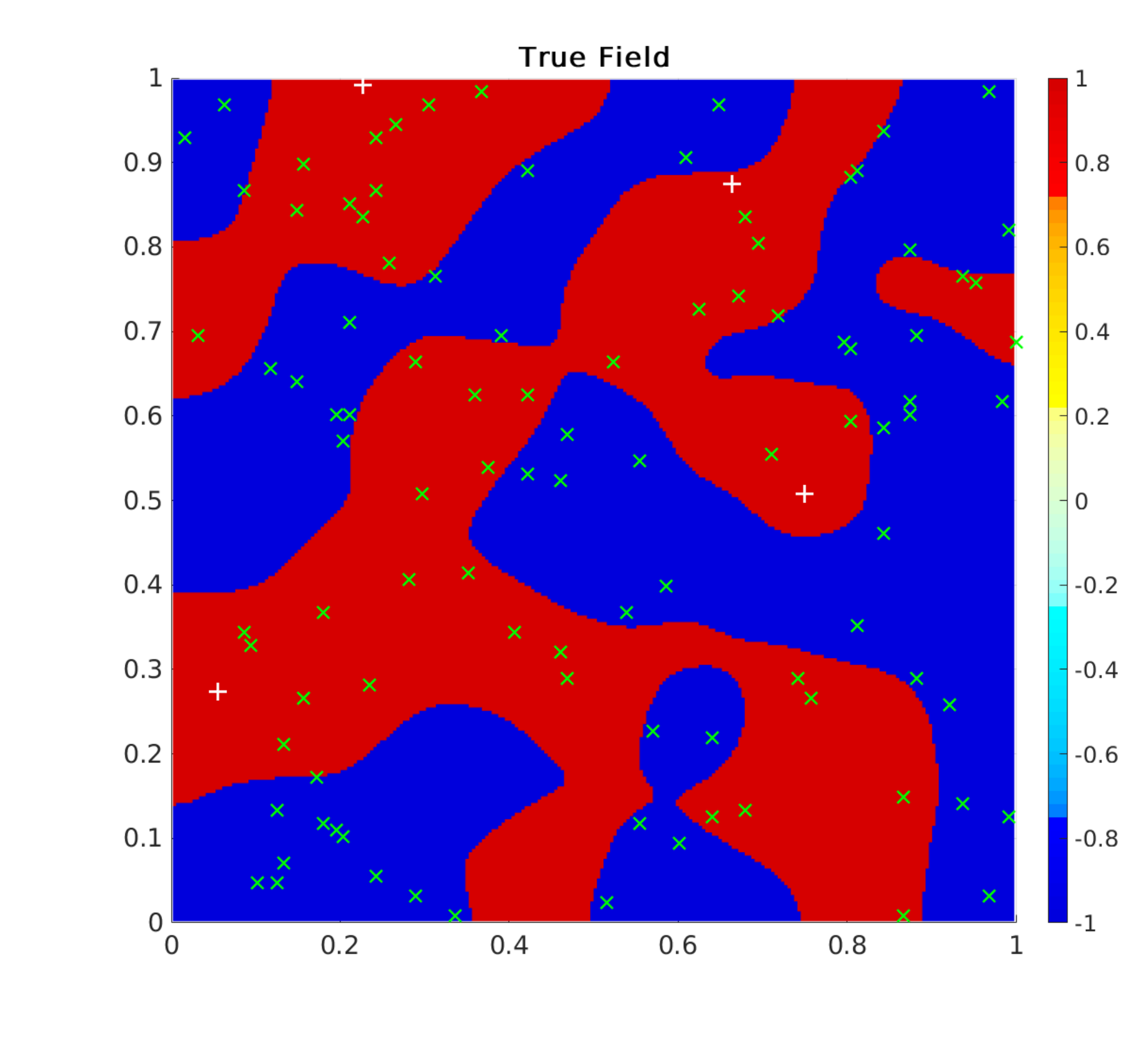}
\includegraphics[width=0.49\textwidth]{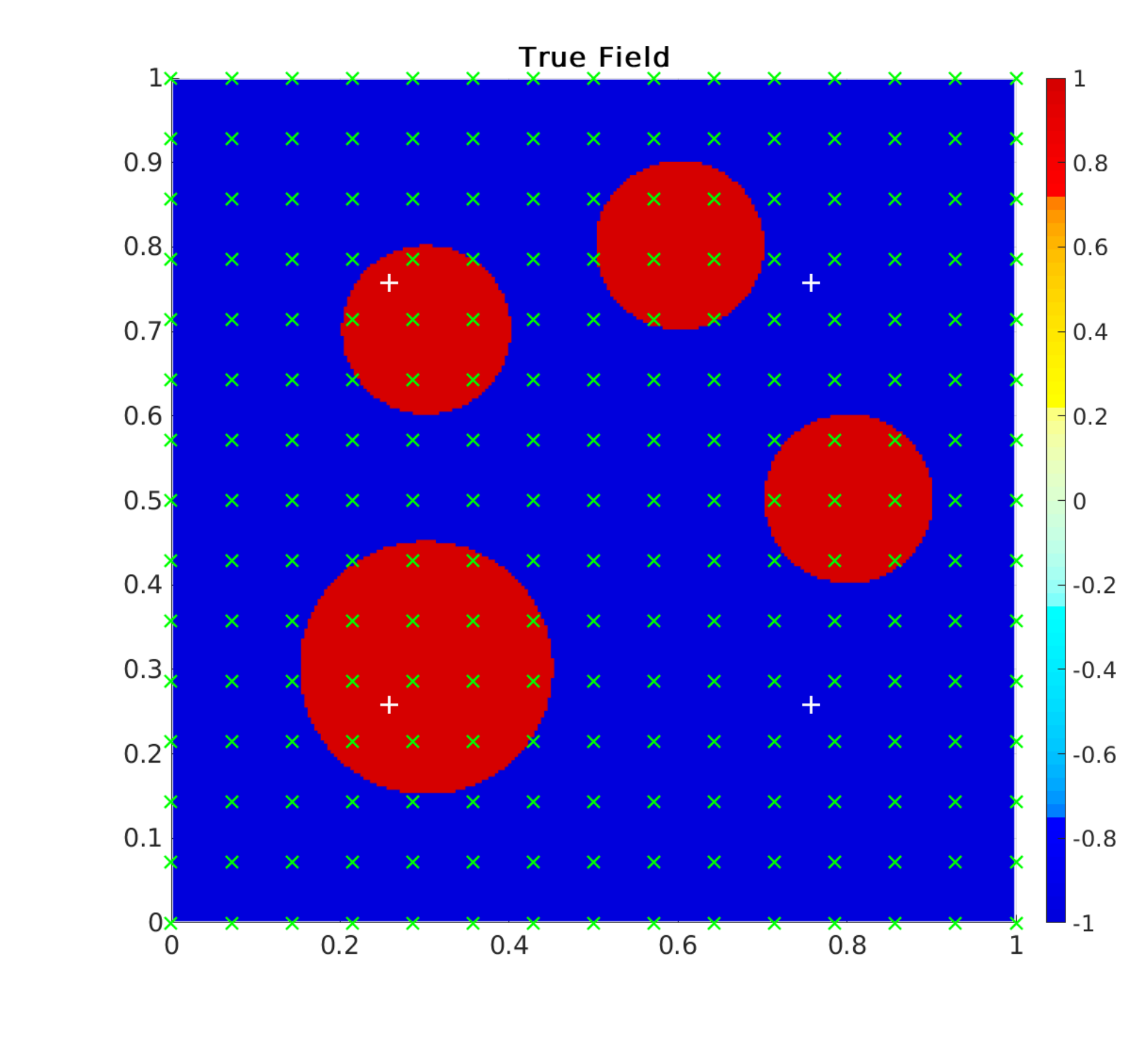}
\end{center}
\caption{The sign of fields $v$, where $\mmm[S](v)$ (for appropriate $u_{\min}$, $u_{\max}$) will be used as a truth for inversion; the field on the left will be referred to as Truth D and the field on the right as Truth E. The sets of observation ($\times$) and source ($+$) points are shown in each figure.}
\label{f:truth_eik}
\end{figure}
For this first test we wish to demonstrate recovery of hyperparameters with the nonlinear eikonal forward model as described in \cref{ssec:hierbayes}. Our domain is given by $\OOmega=[0,1]\times [0,1]$. We choose Truth D as seen in \cref{f:truth_eik}; here the truth has been produced by applying the slowness function $\mmm[S](v^*)$ for an instance $v^*\sim \mu_0(\tau^*)$ with hyperparameter $\tau^*=\exp(6.5)$. We take $u_{\min}=1.0$, and a known contrast $\kappa=0.2$ (thus $u_{\max}=1.2$).

The discretization uses an equidistributed mesh with grid spacing $h=128^{-1}$. To avoid an inverse crime we produce the data on a numerical mesh with spacing $h/2$. We take 4 sources $\{x^m_0\}_{m=1}^4$ and 100 receivers uniformly distributed in the (coarser) discrete domain.     

We work with the parameters same as taken in \cref{ssec:thisone} 
with the exception of setting $c=1$ and, of course, viewing $\tau$
as an unknown. For consistency across different source--receiver combinations we additionally scale the noise by the range of the traveltime observations, so the effective observational noise is $10^{-2}$. We take $\alpha=3$ to ensure finite perimeter.  

For the hyperparameter $\tau$, we choose a lognormal prior to ensure positivity, we take prior $\pi(\log\tau)=\sN(5,2.5)$ and initialize the Markov chain
at $\log(3)$. We use the non-centred pCN-within-Gibbs algorithm  \cite[Algorithm 6.1]{CheDunPapStu19_preprint} described in \cref{ssec:hierbayes} with a random walk Metropolis proposal for the hyperparameter step. The step sizes were chosen to achieve $20\%-25\%$ acceptance rates for $\xi$ and $\tau$.  

The result of the recovery of truth D after $2\times 10^5$ iterations run is displayed in \cref{f:meanslow_eik_rand}. The recovery of the field is fairly 
faithful; we note that information can only be learnt between source and receiver pairs, see Truth D in \cref{f:truth_eik} for their random distribution.  We additionally provide the distribution of the perimeter in \cref{f:interface_eik_rand} with true perimeter marked, and one can see that this falls well within the posterior distribution with high probability mass given to a small
neighborhood of the truth. In \cref{f:loglength_eik_rand} we see the results of learning the hyperparameter distribution. We see the marked truth and the prior and posterior distribution, and again we see the posterior peaks near the true value and gives large mass to a small
neighborhood of the truth. We note that the parameter $\log\tau$ is sampling close to its 
posterior within $10^4$ iterations in this example.

\begin{figure}
\begin{center}
\includegraphics[width=0.49\textwidth]{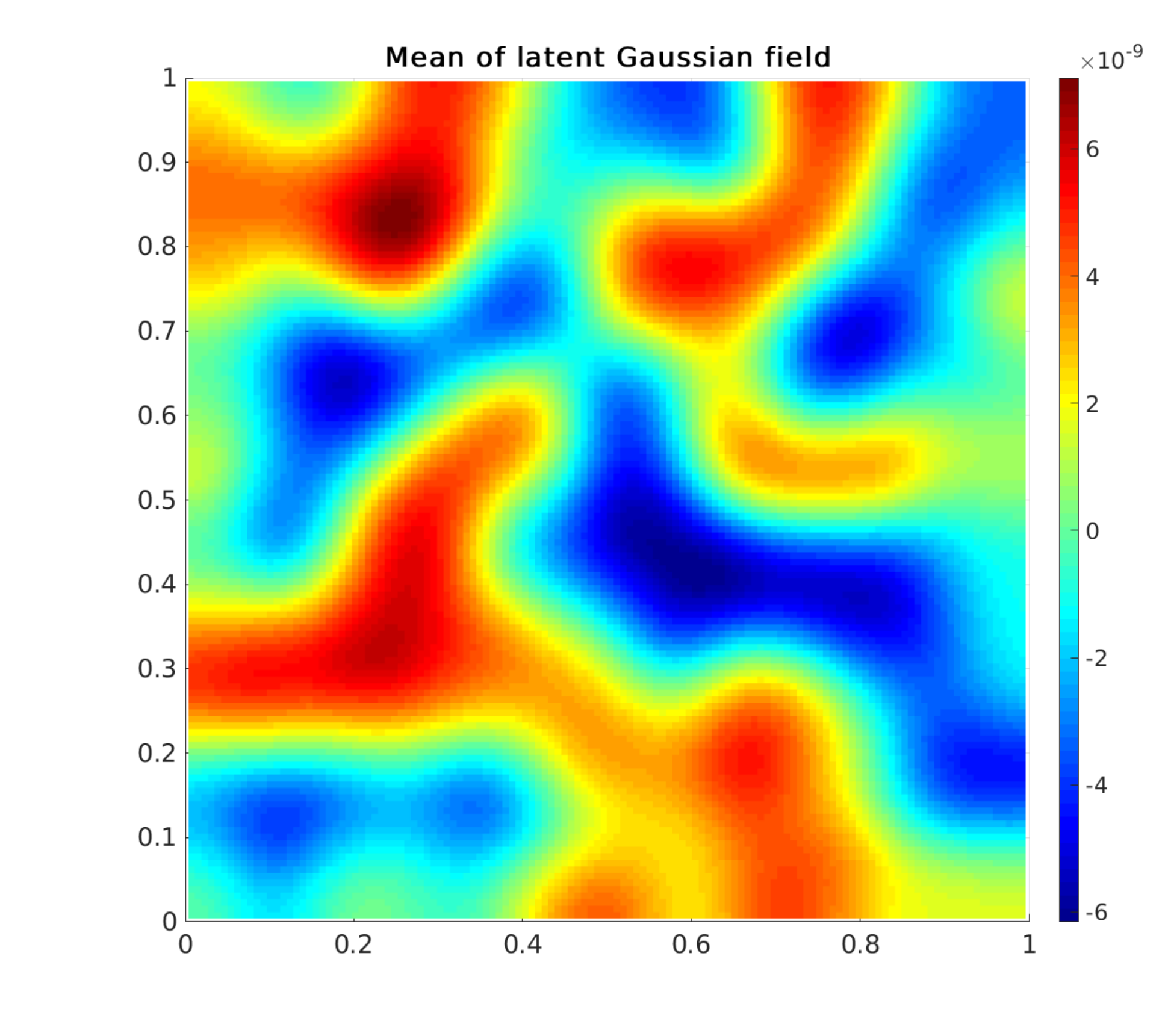}
\includegraphics[width=0.49\textwidth]{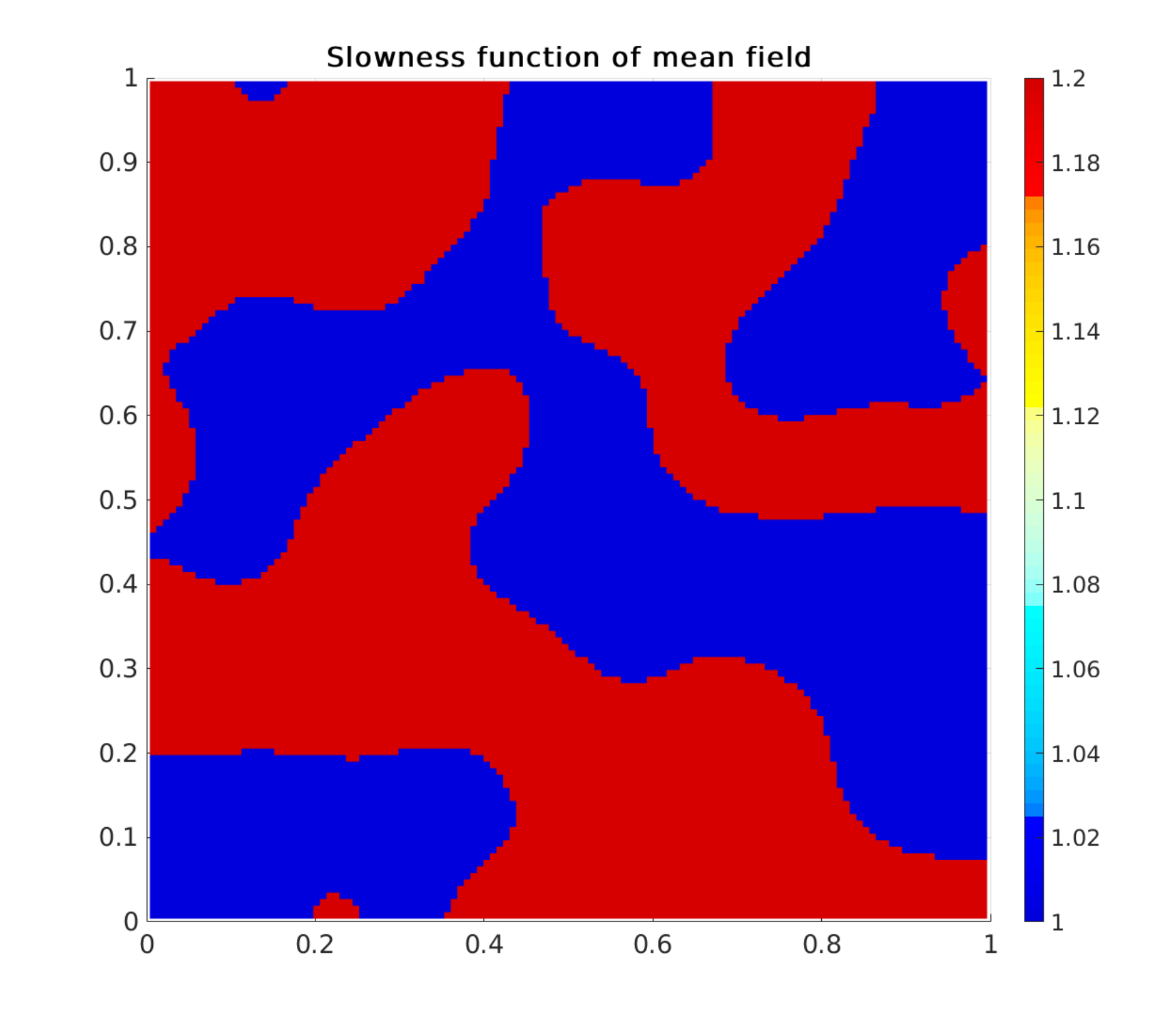}
\end{center}
\caption{Sample means for truth D. The left shows the Monte Carlo approximation of $\mathbb{E}^{\mu^y}(v)$, the underlying continuous field. The right shows the Monte Carlo approximation of $\mmm[S](\mathbb{E}^{\mu^y}(v))$, the thresholded and scaled field.}
\label{f:meanslow_eik_rand}
\end{figure}

\begin{figure}
\begin{center}
\includegraphics[width=0.8\textwidth]{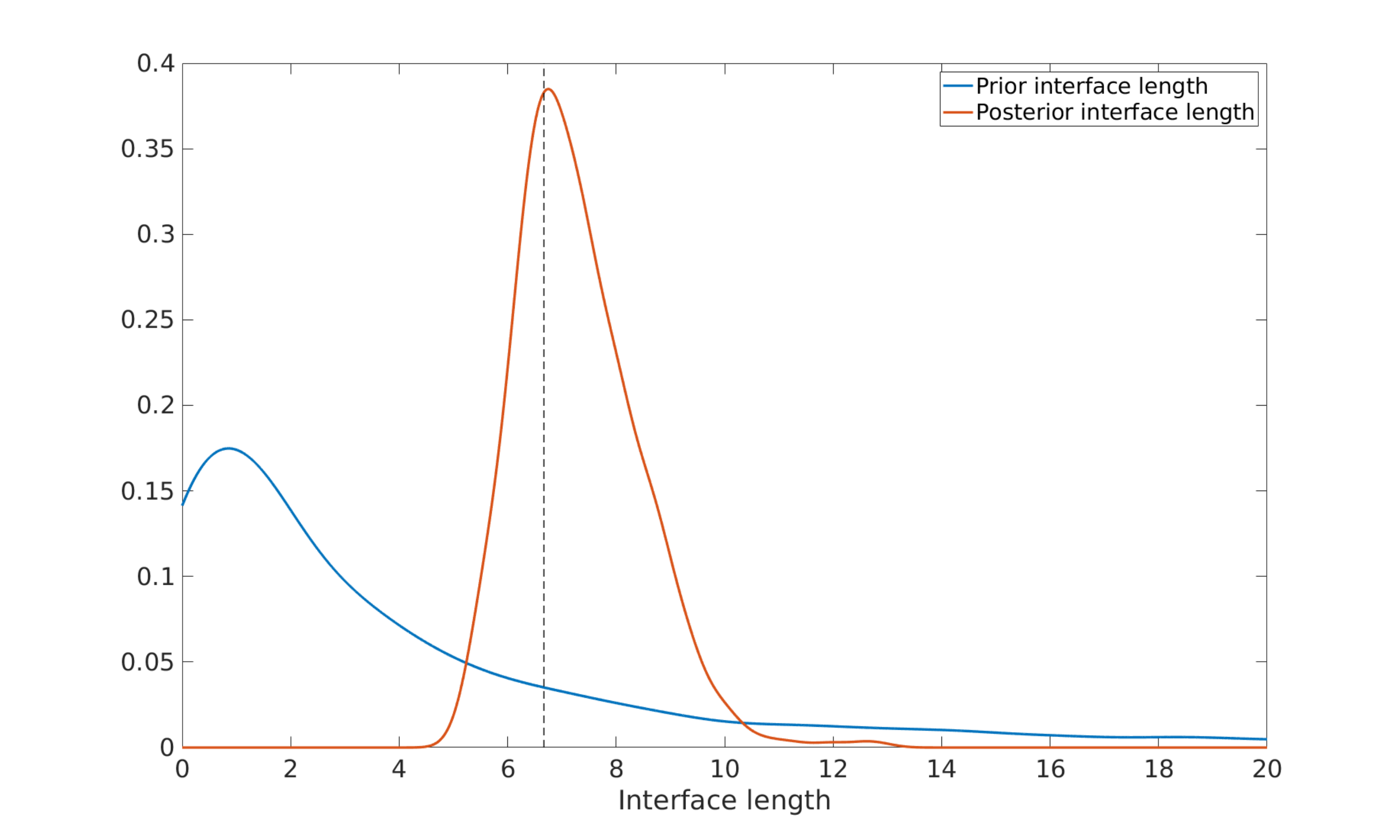}
\end{center}
\caption{The distribution of the perimeter under the prior and posterior distribution for truth D. The vertical dashed line indicates the perimeter of the true field.}
\label{f:interface_eik_rand}
\end{figure}

\begin{figure}
\begin{center}
\includegraphics[width=0.8\textwidth,]{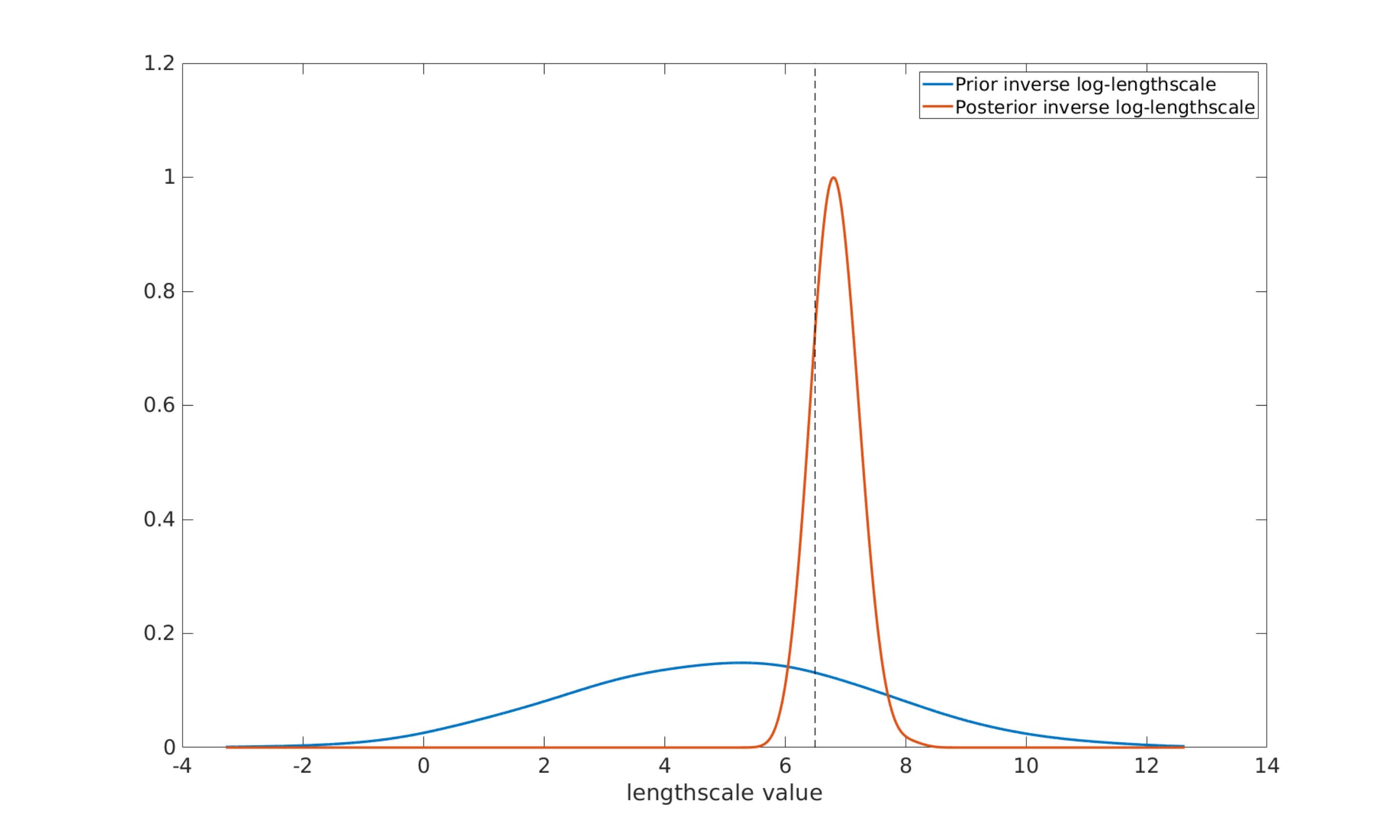}
\end{center}
\caption{The distribution of the logarithm of the hyperparameter $\tau$ under the prior and posterior distribution for truth D. The vertical dashed line indicates the hyperparameter value $\log\tau=6.5$ of the true field.} 
\label{f:loglength_eik_rand}
\end{figure}

\subsection{Numerical Results: Lengthscale And Contrast Hyperparameters}

For this test we demonstrate recovery of the contrast in the
medium. We assume that we are in the situation of performing the inverse eikonal problem where we do not have exact information on the contrast between the binary phases.  We choose Truth E as seen in \cref{f:truth_eik}, 
comprising four circles; three of diameter $0.1$ and one of diameter $0.15$. In the figure we see the choice of four sources, and take $15^2$ equally spaced receivers over the domain. We take $u_{\min}=1.0$, and again we work with the parameters as detailed
in \cref{ssec:thisone} with the exception of setting $c=1$, and 
viewing $\tau$ as unknown and, as in the previous subsection,
we set $\alpha=3$. 
We assume the contrast is lognormal with prior $\pi(\log \kappa)=\sN(\log (0.2), 0.3)$, as we require a positive prior, we also treat $\tau$ as a random unknown quantity, and take a lognormal $\pi(\log \tau)=\sN(5, 2.5)$. We 
initialize the MCMC method to sample these variables at 
$\log\tau=8$ and $\log\kappa=\log(0.1)$. 

Our recovery of Truth E after $2\times 10^5$ iterations is recorded in \cref{f:means_eik_circ}. We see an excellent recovery of the simple geometry, in particular the colorscale shows the approximation of the mean recovered contrast to the true contrast, where $u_{\min}=1,\ u_{\max}=1.2$. This recovery is 
demonstrated in \cref{f:interface_eik_circ} where we see the prior and posterior
densities for the interfacial length. The posterior places most weight in
a small neighbourhood of the truth, and peaks nearby. 

The hyperparameter recovery is displayed in \cref{f:loglen_eik_circ} and \cref{f:logconc_eik_circ}. We see the profiles for prior and posterior for $\log\tau$ in \cref{f:loglen_eik_circ}, demonstrating that the posterior concentrates within a sensible range; note that there is no true $\tau$ for this example.
 In \cref{f:logconc_eik_circ} we see the contrast recovery, and observe that
the posterior places large weight close to the true value. It again takes less 
than $10^4$ iterations for both hyperparameters to draw approximately from
their posterior distributions.

\begin{figure}
\begin{center}
\includegraphics[width=0.49\textwidth]{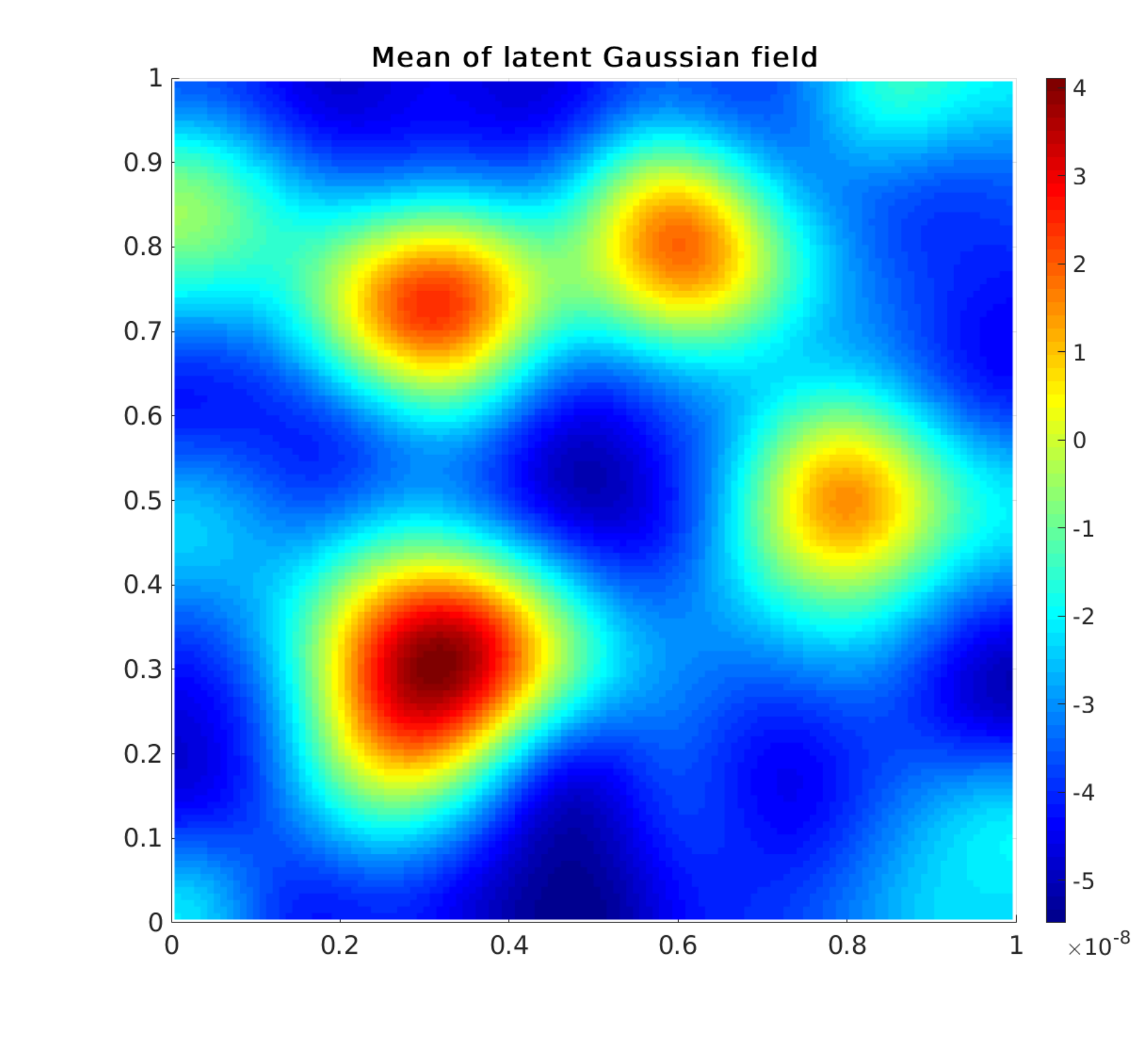}
\includegraphics[width=0.49\textwidth]{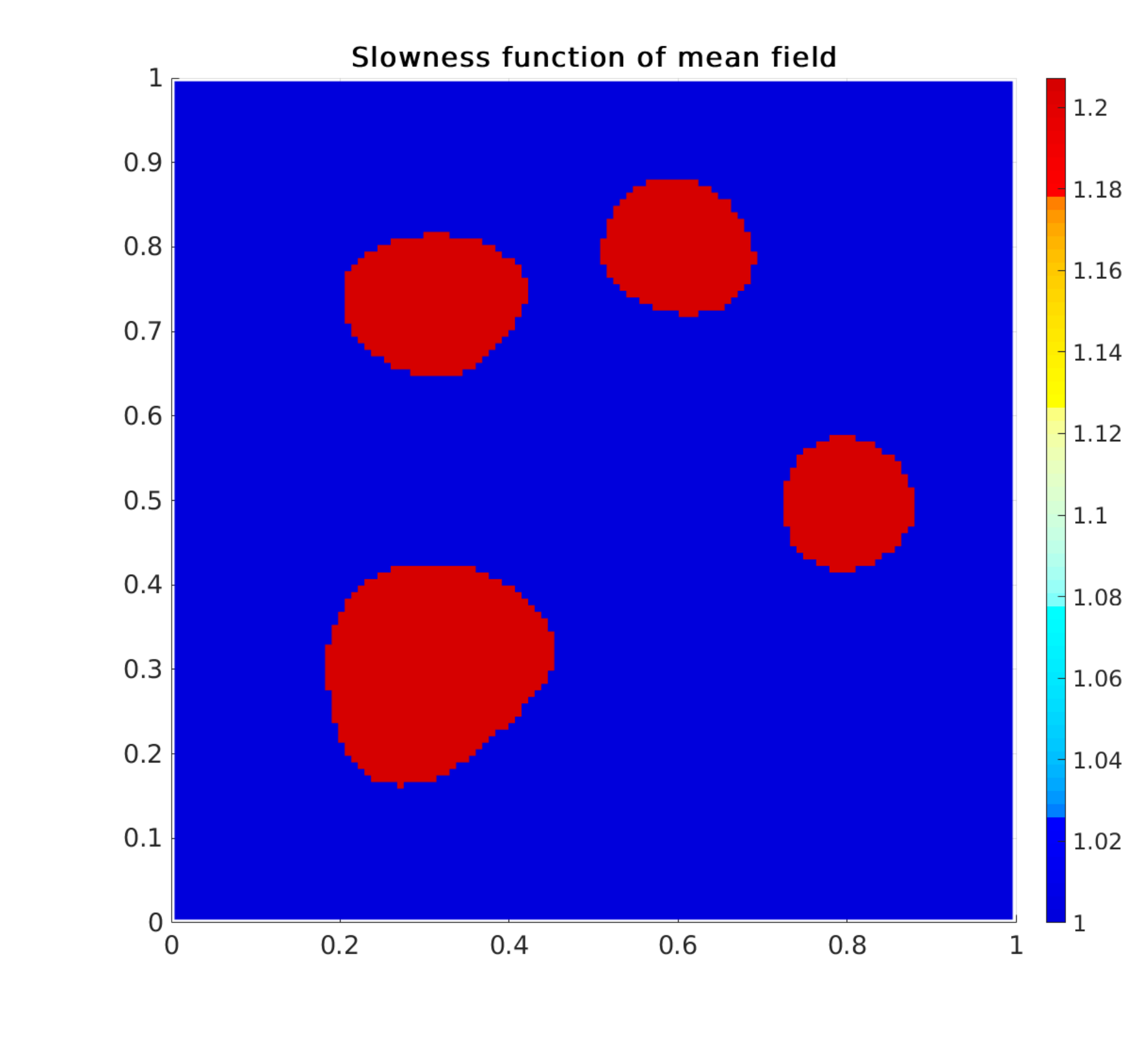}
\end{center}
\caption{Sample means for truth E. The left shows the Monte Carlo approximation of $\mathbb{E}^{\mu^y}(v)$, the underlying continuous field. The right shows the Monte Carlo approximation of $\mmm[S](\mathbb{E}^{\mu^y}(v))$, the thresholded and scaled field.}
\label{f:means_eik_circ}
\end{figure}

\begin{figure}
\begin{center}
\includegraphics[width=0.8\textwidth]{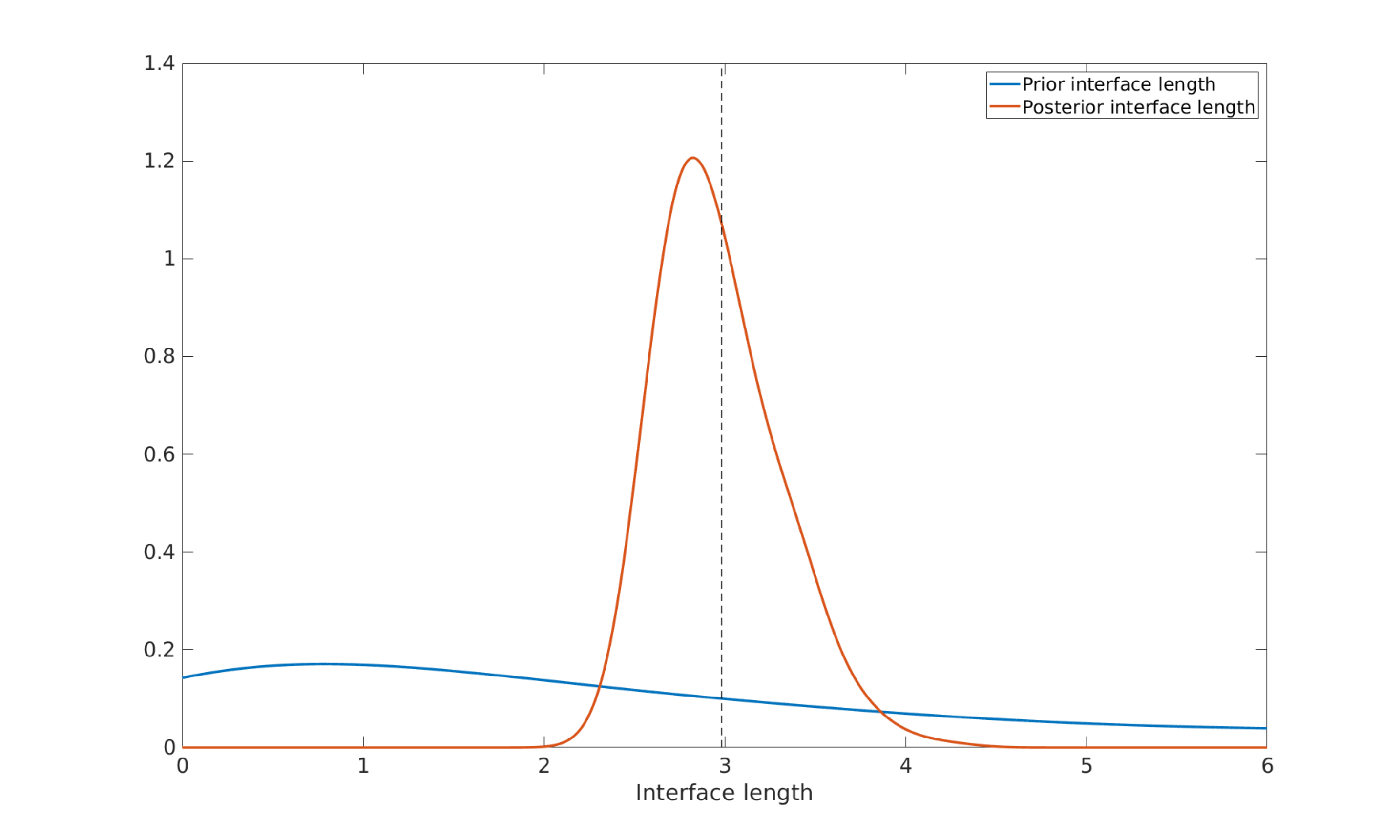}
\end{center}
\caption{The distribution of the perimeter under the prior and posterior distribution for truth E. 
The vertical dashed line indicates the perimeter of the true field.}
\label{f:interface_eik_circ}
\end{figure}

\begin{figure}
\begin{center}
\includegraphics[width=0.8\textwidth,]{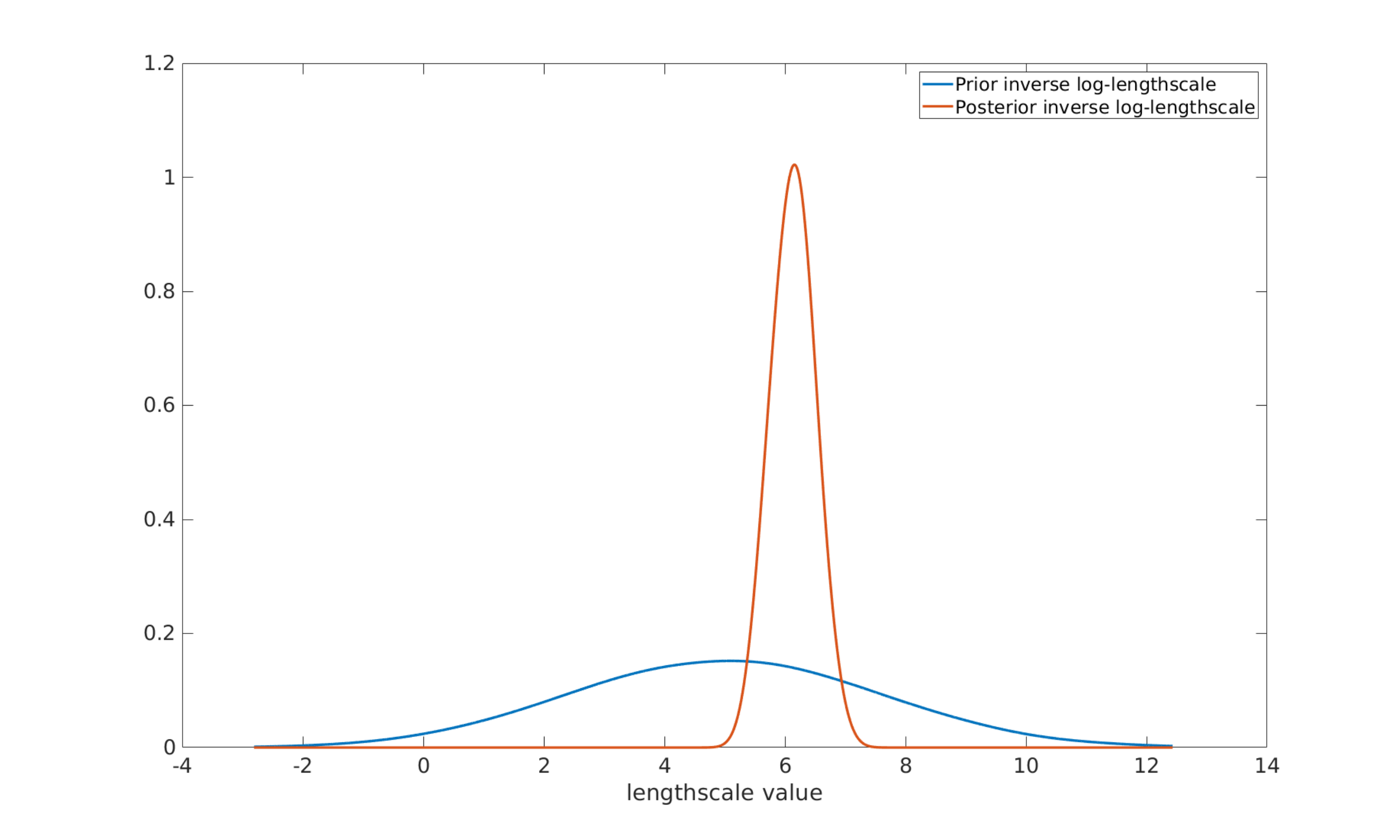}
\end{center}
\caption{The distribution of the logarithm of $\tau$ under the prior and posterior distribution for 
truth E.}
\label{f:loglen_eik_circ}
\end{figure}

\begin{figure}
\begin{center}
\includegraphics[width=0.8\textwidth,]{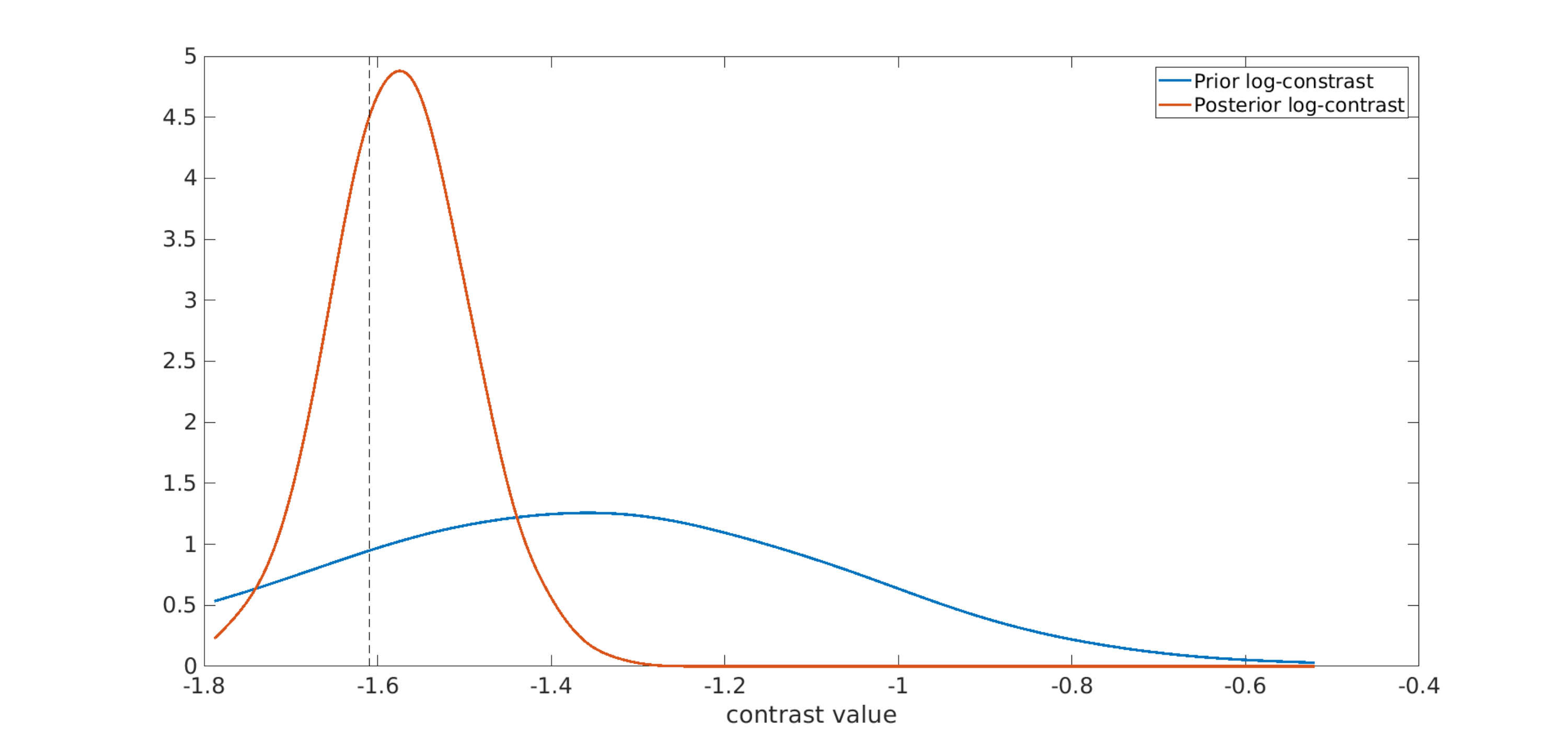}
\end{center}
\caption{The distribution of the logarithm of $\kappa$ under the prior and posterior distribution for truth E. 
The vertical dashed line indicates true value $\log\kappa=\log(0.2)$.}
\label{f:logconc_eik_circ}
\end{figure}

\section{Conclusions}\label{sec:C}

The paper investigates the reconciliation of perimeter and 
Bayesian regularization for the reconstruction of functions
with interfaces, from direct or indirect noisy measurements.
Three approaches are studied: Formulation 1 is based on Bayesian level method; Formulation 2 is based on Bayesian
phase-field regularization;
and Formulation 3 is based on Gaussian process regression and
classification.

By studying a class of linear inverse problems
we show that Formulation 2 exhibits perimeter regularization
in the context of its MAP estimator, but not at the level
of samples from the posterior distribution.
Formulation 1 exhibits perimeter regularization at the
level of individual samples from the posterior; there is no MAP estimator in this context. Both Formulations 
1 and 2 require careful choices of constants in construction of the prior, 
but Formulation 2 is far more constrained in this regard. Furthermore,
as a consequence of these constraints, Formulation 2 exhibits
a measure concentration phenomenon meaning that MCMC based algorithms 
using Formulation 1 are considerably faster than those based on Formulation 2. 
Formulation 3
is competitive with Formulation 1 in terms of both sample properties
and speed, but does not generalize beyond linear problems.
We study Formulation 1 for a nonlinear inverse problem, demonstrating
that it is effective in this context and, in addition, showing
how hierarchical methods may be used to learn model hyper-parameters
appearing in the prior.

The ideas in this paper can be combined in different ways: 
the methodology may be extended beyond binary-valued functions
to a variety of piecewise continuous problems; or
other limiting functionals could be contemplated, such as Mumford-Shah \cite{helin2011hierarchical}; 
and other smoothed thresholding functions could be contemplated within 
the level set method, such as the double-obstacle approximation 
to the signum function \cite{blowey1993curvature,blowey1994phase}. 
The success of the Bayesian level set method suggests that
further analysis of it, as well as its deployment in new application domains, 
would be very valuable.

\bibliographystyle{siamplain}

\appendix
\section{Proofs Of Main Results} 
\label{sec:A}

\noindent {\em Proof of \cref{prop:2}.}
Throughout this proof $C$ is a universal constant whose value may change between
occurrences. To apply Theorem 4.12 from \cite{DLSV13}, we need to show that the 
function $\Phi(\cdot,y)$ is bounded from below, is locally bounded from above and is locally Lipschitz.
We note that $\Phi(\cdot,y)$ is always non-negative so is bounded from below. 
If $\|u\|_X=\max_{x\in \bar\OOmega}|u(x)|\le \rho$ then we may bound $|\Phi(u,y)|$ by a constant
depending on $\rho$ i.e. $\Phi(\cdot,y)$ is locally bounded. For the local Lischiptz continuity, we have
\beqas
\Phi(u,y)-\Phi(v,y)=\frac{r}{4\ep^b}\int_\OOmega(2-u(x)^2-v(x)^2)(u(x)+v(x))(u(x)-v(x))\,\dee x+\\
\frac{1}{2\ep^{2c}}\langle\GGamma^{-\frac{1}{2}}(2y-Ku-Kv),\GGamma^{-\frac{1}{2}}K(v-u)\rangle
\eeqas
{Assume that $\|u\|_X\le \rho$ and $\|v\|_X\le \rho$. Then, since $K$ is a bounded
linear operator on $L^1(D),$} 
\beqas
|\Phi(u,y)-\Phi(v,y)|\le C\int_\OOmega|u(x)-v(x)|dx+C|K(v-u)|\\
\le C\|u-v\|_{L^1(\OOmega)} \le C|\OOmega|^{1/2}\|u-v\|_{L^2(\OOmega)} \le C\|u-v\|_X.
\eeqas
The desired result follows.
\eproof

\hspace{0.2in}

\noindent {\em Proof of \cref{thm:3}.}
We adapt the proof of Hilhorst et al. to allow for periodic boundary conditions and
the additional $L^2$ norm appearing in the functional to be infimized.  
From Hilhorst et al., we have that if $\ue\to u$ in $L^1(\OOmega)$ then 
\begin{eqnarray*}
\liminf_{\ep\to 0}I^\ep(\ue)&\ge& \liminf_{\ep\to 0}\int_D\left(\frac12{\delta\ep^3}|\triangle\ue|^2+\frac12\delta q\ep|\nabla\ue|^2
+\frac{r}{4\ep}\bigl(1-\ue(x)^2\bigr)^2\right)\dee x\\
&&+\frac{1}{2}|\GGamma^{-\frac{1}{2}}(y-K\ue)|^2\\
&\ge& I_0^\delta(u).
\end{eqnarray*}
Now we show that for each $u\in L^1(\OOmega)$, there is a sequence 
$\{\ue\}\subset H^2_\#(\OOmega)$ which converges strongly to $u$ in $L^1(\OOmega)$ such that 
$\limsup_{\ep\to 0} I^\ep(\ue)\le I_0^\delta(u)$. We first review the main points 
in the proof of Hilhorst et al. for functions $u\in H^2(\OOmega)$. Considering the 
case $I^\delta(u)<\infty$, without loss of generality, we assume that 
\[
u=\one_Q-\one_{\IR^d\setminus Q}
\]
where $Q$ is a bounded domain, with $\partial Q\in C^\infty$ and $Q\subset\subset \OOmega$.
The sign distance function is defined as
\[
d(x)=\left\{
\begin{array}{rl}
+\inf_{y\in\partial Q}|x-y|\ \mbox{if}\ x\in Q\\
-\inf_{y\in\partial Q}|x-y|\ \mbox{if}\ x\notin Q
\end{array}
\right.
\]
Let $N_h$ be an $h$ neighbourhood of $\partial Q$ (we choose $h$ so that $h$ is less than 
the distance between $\partial Q$ and $\partial \OOmega$.) We choose a function $\eta\in C^2(\bar\OOmega)$ 
such that $\eta(x)=d(x)$ for $x\in N_h$, $\eta(x)\ge h$ when $x\in Q\setminus N_h$ and $\eta(x)\le -h$ 
when $x\in\OOmega\setminus(Q\bigcup N_h)$. Let $U$ be an odd minimizer of the functional 
$e^\delta(U)$ with $\lim_{t\to\infty}U(t)=1$ and $\lim_{t\to-\infty}U(t)=-1$. We let
\[
\ue=U\left(\frac{\eta(x)}{\ep}\right).
\]

We note that $\ue(x)$ is uniformly bounded pointwise and $\ue(x)\to u(x)$ for all $x\in D$. 
From the Lebesgue dominated convergence theorem, $\ue\to u$ in $L^1(D)$ and in $L^2(D)$. Thus
\[
\lim_{\ep\to 0}|\GGamma^{-\frac{1}{2}}(y-K\ue)|^2=|\GGamma^{-\frac{1}{2}}(y-Ku)|^2.
\]
and, since $a>0$, 
\[
\lim_{\ep\to 0}\int_D\delta\tau^2\ep^a\ue(x)^2\,\dee x=0. 
\]
To show that $\lim_{\ep\to 0}I^\ep(u)=I^\delta_0(u)$, we follow the approach of Hilhorst et al.. 
The integral 
\[
\int_D\left(\frac12{\delta\ep^3}|\triangle\ue|^2+\frac12\delta q\ep|\nabla\ue|^2
+\frac{r}{4\ep}\bigl(1-\ue(x)^2\bigr)^2\right)\dee x
\]
is written as 
\begin{eqnarray*}
\int_{D\setminus N_h}\left(\frac12{\delta\ep^3}|\triangle\ue|^2+\frac12\delta q\ep|\nabla\ue|^2
+\frac{r}{4\ep}\bigl(1-\ue(x)^2\bigr)^2\right)\dee x\\
+\int_{N_h}\left(\frac12{\delta\ep^3}|\triangle\ue|^2+\frac12\delta q\ep|\nabla\ue|^2
+\frac{r}{4\ep}\bigl(1-\ue(x)^2\bigr)^2\right)\dee x.
\end{eqnarray*}
Using the exponential decay of $U, U'$ and $U''$ at $\infty$ and $-\infty$, we 
deduce that the integral over $\OOmega\setminus N_h$ goes to $0$ when $\ep\to 0$ 
(note that $|\eta(x)|/\ep>h/\ep$ which goes to $\infty$ when $\ep\to 0$ for $x\in\OOmega\setminus N_h$).  
The integral over $N_h$ is shown to converge to  $I^\delta(u)$ when $\ep\to 0$. 

To adapt this proof of Hilhorst et al. to functions with periodic boundary condition on $\OOmega$, we 
only need to choose the function $\eta$ so that $\eta$ is periodic and $\eta(x)\ge h$ 
when $x\in Q\setminus N_h$ and $\eta(x)\le -h$ when $x\in \OOmega\setminus (Q\bigcup N_h)$. 
Such a function can be constructed as follows. Let $\psi(x)\in C^\infty_0(\OOmega)$ be such 
that $\psi(x)=1$ when $x$ is in a neighbourhood of $Q\bigcup N_h$, and $0\le \psi(x)\le 1$ for 
all $x\in \OOmega$. Let $\eta_1(x)$ be a smooth periodic function with $\eta_1(x)\le -h$ 
for all $x\in\OOmega$. Using the function $\eta$ of Hilhorst et al., we define a new function 
\[
\bar\eta(x)=\psi(x)\eta(x)+(1-\psi(x))\eta_1(x).
\]  
The function $\bar\eta(x)$ satisfies the requirement.
\hfill$\Box$

\end{document}